\documentclass[12pt,a4paper]{amsart}
\usepackage[utf8]{inputenc}
\usepackage[T1]{fontenc}
\usepackage[foot]{amsaddr}
\usepackage{amsmath}
\usepackage{amssymb}
\usepackage{amsfonts}
\usepackage{amsthm}
\usepackage[all]{xy}
\usepackage{tikz}
\usepackage{hyperref}
\usepackage[shortlabels]{enumitem}
\usepackage{multicol}
\usepackage[enableskew]{youngtab}
\usepackage{ytableau}
\usepackage{multicol}
\usepackage{bbold}
\usepackage{quiver}
\usepackage{geometry}
\usepackage{tabularray}
\usepackage{geometry}
\usepackage{apptools}
\usepackage{xcolor}

\theoremstyle{plain}
\newtheorem*{atw*}{Theorem}
\newtheorem{atw}{Theorem}[section]
\newtheorem*{lemma*}{Lemma}
\newtheorem{lemma}[atw]{Lemma}
\newtheorem{pro}[atw]{Proposition}
\newtheorem*{pro*}{Proposition}
\newtheorem{cor}[atw]{Corollary}
\newtheorem*{cor*}{Corollary}

\theoremstyle{definition}
\newtheorem{adf}[atw]{Definition}
\newtheorem*{adf*}{Definition}

\theoremstyle{remark}
\newtheorem{rem}[atw]{Remark}
\newtheorem*{rem*}{Remark}
\newtheorem{ex}[atw]{Example}
\newtheorem*{ex*}{Example}

\def\H{\operatorname{Hilb}}
\newcommand{\Hilb}[1]{\H_{{#1}}}

\newcommand{\Rep}{\operatorname{Rep}}
\newcommand{\qq}{\mathfrak q}
\newcommand{\ch}{\operatorname{ch}}
\newcommand{\eu}{\operatorname{eu}}
\newcommand{\td}{\operatorname{td}}
\newcommand{\ZZ}{\mathbb{Z}}
\newcommand{\NN}{\mathbb{N}}
\newcommand{\QQ}{\mathbb{Q}}

\newcommand{\CC}{\mathbb{C}}
\newcommand{\Aa}{\mathbb{A}}
\newcommand{\TT}{\mathbb{T}}

\newcommand{\Tt}{{\mathbb{T}_y}}
\newcommand{\HH}{\mathbb{H}}
\newcommand{\Pcoh}[2]{{P^{#2}_{#1}}}
\newcommand{\PKT}[2]{{\mathcal{P}^{#2}_{#1}}}
\def\O{\mathcal{O}}
\def\L{\mathcal{L}}
\def\V{\mathcal{V}}
\def\Q{\mathcal{Q}}

\newcommand{\xtto}[1]{\stackrel{#1}{\longrightarrow}}

\newcommand{\mono}{\hookrightarrow}
\newcommand{\onto}{\twoheadrightarrow}

\newcommand{\rank}{\operatorname{rank}}
\newcommand{\supp}{\operatorname{supp}}
\newcommand{\coh}{\operatorname{H}}
\newcommand{\Pic}{\operatorname{Pic}}
\DeclareMathOperator{\KTh}{K}
\DeclareMathOperator{\Hom}{Hom}

\title{Nakajima's creation operators and the Kirwan map}
\author{Jakub Koncki}
\address{Institute of Mathematics, Polish Academy of Sciences, Poland}
\address{Institute of Mathematics, University of Warsaw, Poland}
\email{j.koncki@mimuw.edu.pl}

\author{Magdalena Zielenkiewicz}
\address{Institute of Mathematics, University of Warsaw, Poland}
\email[Corresponding author]{magdaz@mimuw.edu.pl}

\begin{document}
\ytableausetup{smalltableaux}
\begin{abstract}
	 	We consider the Hilbert scheme of points in the affine complex plane. We find explicit formulas for the Nakajima's creation operators and their $\KTh$-theoretic counterparts in terms of the Kirwan map. We obtain a description of the action of Nakajima's creation operators on the Chern classes of the tautological bundle. 
\end{abstract}
\maketitle

\section{Introduction}
The Hilbert scheme of $n$ points in the complex plane $\Hilb{n} := \Hilb{n}(\Aa_{\CC}^2)$ is a quasi-projective variety parametrizing zero-dimensional subschemes of degree $n$ in the complex affine plane $\Aa_{\CC}^2$. It is a well-studied object, and a lot is known about its geometry. Many results on the Hilbert scheme of points in the plane have non-obvious applications elsewhere, from geometry \cite{Boissiere2, NY} to algebra and representation theory \cite{GorS1, Wei, BG}, mathematical physics \cite{Carl, Gorsky} and combinatorics \cite{Haiman2, Haiman}. While it has been intensively studied ever since the 1960's, it is still an active research topic \cite{OR, Pandha, Przez} and there remain open questions to be addressed. \\

Many connections between the Hilbert scheme of points in the plane and other branches of mathematics have been unearthed by the Nakajima--Grojnowski \cite{Nak1, Gro} milestone result which endowed the sum of all the rational homology groups with a structure of a representation of the Heisenberg algebra. Such a structure can be obtained by constructing certain operators on homology, which satisfy the commutation relations of the Heisenberg algebra. In Nakajima's construction such operators, called creation and annihilation operators, are given by correspondences defined by subvarieties in the nested Hilbert scheme. Originally defined as operators in homology, one can extend them to equivariant cohomology using a similar construction, see \cite{Evain}. The situation is more complicated for $\KTh$-theory, as pointed out already by Nakajima (see \cite[Question 8.35]{Nak1}). There are, however, generalizations of the Nakajima--Grojnowski result to $\KTh$-theory.
In \cite{FT}, the actions of the shuffle and the Ding-Iohara algebras are constructed on the equivariant $\KTh$-theory. In turn, in \cite{SV}, a subalgebra of the convolution algebra in the equivariant $\KTh$-theory of $\Hilb{n}$ is studied and is identified with the elliptic Hall algebra. \\

Most of this paper is devoted to $\KTh$-theory of the Hilbert scheme of points in the plane, and the final chapters present implications of the $\KTh$-theoretical results to cohomology. While equivariant cohomology of the Hilbert scheme of points has been extensively studied (see e.g. \cite{ES,Lehn,Evain,Cheah2, GS}), there remain questions which do not have a satisfactory answer. One such question, stated in \cite[Question 9.6]{Nak1}, is how do the creation operators act on the Chern classes of the tautological bundle on the Hilbert scheme. We present an answer to this question. A similar problem is considered in \cite{Lehn}. There Chern classes of the tautological bundle are expressed in Nakajima's basis. Results presented here and in \cite{Lehn} are connected, yet not directly comparable, see Remark \ref{rem:Lehn} for a detailed discussion. \\

Our result is an expression for the action of the Nakajima's operators $\qq_i$ on the characteristic classes of the tautological bundle $\V_n$ in terms of the Kirwan map
$$\kappa: \ZZ[x_0, \dots, x_{n-1}]^{\Sigma_{n}} \to \coh^{\ast}(\Hilb{n})\,.$$
Here $\Sigma_{n}$ denotes the symmetric group and the map $\kappa$ sends variables $x_i$ to the Chern roots of the tautological bundle $\V_n$. Chern classes of $\V_n$ are images of the elementary symmetric polynomials.
We prove a formula for their images under $\qq_1$, expressed in the Chern classes of the tautological
bundle $\V_{n+1}$ as well as the power-sum polynomials of its Chern roots $\Pcoh{k}{n+1}$

$$\qq_1(c_k(\V_n))=\sum_{m=0}^k (-1)^m(m+1)\cdot c_{k-m}(\V_{n+1})\cdot \Pcoh{m}{n+1}\,. $$
It turns out that the power sum basis of the symmetric polynomial ring is better suited for our formulas. Our main result is presented below.
\begin{atw*}[\ref{tw:qm}]
	Let $\lambda=(\lambda_1,\dots,\lambda_l)$ be a sequence of nonnegative integers and $m$ a positive integer. For a subset $A\subseteq\{1,\dots,l\}$, let $\lambda_A$ be the sequence obtained from $\lambda$ by removing indices corresponding to elements of $A$ and
	$l(A):=\sum_{i\in A}\lambda_i\,.$ Then
	\begin{align*}
		\qq_m(\Pcoh{\lambda}{n})=(-1)^{m+1}\cdot\sum_{A\subseteq\{1,\dots,l\}}(-m)^{|A|}\cdot(l(A)+m)\cdot\Pcoh{{\lambda_A}}{n+m}\cdot\Pcoh{l(A)+m-1}{n+m}\,.
	\end{align*}
In particular for a nonegative integer $k\ge 0$ we have
\begin{align*}
	\qq_m(\Pcoh{k}{n})=(-1)^{m+1}\cdot m\cdot\left(\Pcoh{k}{n+m}\cdot\Pcoh{m-1}{n+m}-(m+k)\cdot\Pcoh{k+m-1}{n+m}\right)\,.
\end{align*}
\end{atw*}
To calculate images of the Chern classes one needs to change basis in the symmetric polynomial ring.
For the operator $\qq_1$ we obtain an equivariant version (with respect to the one dimensional coordinate torus) of the above formulas.

Moreover, we prove their K-theoretical analogue. Let $\qq^{\KTh}_{1,m}$ be the creation operators in the Ding-Iohara algebra, or the elliptic Hall algebra (see Section \ref{s:Nakoperators}).
\begin{pro*}[\ref{pro:qK}]
	The following holds in nonequivariant K-theory
		$$\qq^{\KTh}_{1,m}(\PKT{k}{n})=\PKT{k}{n+1}\cdot\big(m\PKT{m}{n+1}-(m-1)\PKT{m-1}{n+1}\big)-(k+m)\PKT{k+m}{n+1}+(k+m-1)\PKT{k+m-1}{n+1}\,.$$ 
\end{pro*}

We first prove theorem \ref{tw:qm} for the operator $\qq_1$ (Proposition \ref{tw:q1}). For higher $\qq_i$,  \cite[Theorem 34]{Evain}, \cite[Proposition 3.12]{Lehn} gives a recursive formula using the auxiliary operator~$\rho$.
We prove an explicit formula describing the action of $\rho$ on the power sum basis in Proposition \ref{pro:rho}. This allows to extend our result on $\qq_1$ to all the operators~$\qq_m$. \\
	
The tautological bundle $\Q_n$ on the nested Hilbert scheme $\Hilb{n,n+1}$ plays a crucial role in our proofs. The operator $\qq_1$ is fully determined by pushforwards of certain classes of this bundle. We study these classes in the torus equivariant $\KTh$-theory. We use equivariant Grothendieck--Riemann--Roch theorem \cite{EGGRR} to transport our results to the torus equivariant cohomology. As a corollary we obtain formulas in the nonequivariant theories.
\begin{atw*}[\ref{tw:push} and \ref{tw:pushH}]
	Let $\pi: \Hilb{n,n+1}(\Aa_{\CC}^2) \to \Hilb{n+1}(\Aa_{\CC}^2)$ be the projection. For an integer $m$  We have
	$$\pi_*[\Q^{m}_n]=
	\frac{1-t^{-m}}{1-t^{-1}}\cdot \PKT{m}{}
	-\frac{1-t^{-(m-1)}}{1-t^{-1}}\cdot\frac{1}{t}\cdot\PKT{m-1}{}\in \KTh_{\Tt}(\Hilb{n+1}(\Aa_{\CC}^2))\,.$$
	 Moreover
	$$\pi_*\left(c_1(\Q_n)^m\right)=\sum_{k=0}^{m}a_{k,m}\cdot t^{m-k}\cdot \Pcoh{k}{n+1} \in \coh_{\Tt}^{2m}(\Hilb{n+1}(\Aa_{\CC}^2)) \,, $$
	where the $a_{k,m}$ are the coefficients of the polynomial
	$$\sum_{k=0}^{m}a_{k,m}\cdot x^k=x^m(x+1)-(x-1)^mx\,. $$
\end{atw*}
\begin{cor*}
	For an integer $m\ge 1$ we have
	\begin{align*}
		\pi_*[\Q^{m}_n]&=m\cdot\PKT{m}{}-(m-1)\cdot\PKT{m-1}{}\in \KTh(\Hilb{n+1}(\Aa_{\CC}^2))\,,\\
		\pi_*\left(c_1(\Q_n)^m\right)&=(m+1)\cdot \Pcoh{m}{n+1} \in \coh^{2m}(\Hilb{n+1}(\Aa_{\CC}^2)) \,.
	\end{align*}
\end{cor*}
~\\

The standard action of the two-dimensional torus $\TT = (\CC^{\ast})^2$ on $\CC^2$ induces an action on the Hilbert scheme $\Hilb{n}$. The fixed point set of this action is finite, which allows for the use of localization techniques in the study of  invariants such as equivariant cohomology or $\KTh$-theory. Moreover, since the fixed points of the action are indexed by partitions, such computations may be used to approach conjectures in combinatorics, as in \cite{Haiman}, or give new descriptions of known symmetric functions such as Jack polynomials as in \cite{Nak2}. 
The torus action is not very well-behaved from a certain point of view, for example it does not yield a GKM-variety structure, but it is often powerful enough to allow for explicit computations of algebraic invariants such as rational cohomology or $\KTh$-theory classes, which is the approach we use in this paper. This approach has been applied by many authors; see e.g. \cite{Boissiere, Evain} for cohomological computations, \cite{SV,FT} for $\KTh$-theory and \cite{Smirnov} for elliptic cohomology. \\

The structure of this paper is as follows. Sections \ref{s:partitions}--\ref{s:Nakoperators} provide background. Section \ref{s:partitions} introduces notation for partitions, Young diagrams and symmetric functions, and defines some combinatorial constructions which we use later. In Section \ref{s:Ktheory} we recall the basics of equivariant $\KTh$-theory for torus actions. Section \ref{prel:hilb} provides some introductory information on the Hilbert scheme of points in the plane, including the basics of its equivariant cohomology and $\KTh$-theory. In Section \ref{s:nested} we give a combinatorial description of the tangent space to the nested Hilbert scheme of points in the plane, and explain how to compute the push-forward in equivariant $\KTh$-theory using localization. Section \ref{s:Nakoperators} introduces the creation operators and their $\KTh$-theoretic analogues. These analogues are defined using a line bundle $\Q$ on the nested Hilbert scheme, and one of the key technical results is a computation of its push-forward along the projection to $\Hilb{n+1}(\Aa^2_{\CC})$. This is done in Theorem \ref{tw:push}. The core of this paper are Sections \ref{s:Kwyniki}--\ref{s:Nakwyniki}. In Section \ref{s:Kwyniki} we compute the equivariant and nonequivariant $\KTh$-theoretic push-forward of the line bundle $\Q$. In Section \ref{s:Adams} we use Adams operators to compute pushforwards of its tensor powers.
Section \ref{s:Hwyniki} presents the cohomological equivalents of the results obtained in Sections \ref{s:Kwyniki} and \ref{s:Adams}.
Section \ref{s:Nakwyniki} puts together the results of the preceding sections to provide a description of the action of the  creation operators on the Chern classes of the tautological bundle on~$\Hilb{n}(\Aa^2_{\CC})$. \\

\paragraph{\bf Acknowledgements:} MZ is supported by NCN grant SONATA 2020/39/D/ST1/00132. JK is supported by NCN grant SONATINA 2023/48/C/ST1/00002. Both authors are grateful to Joachim Jelisiejew and Andrzej Weber for helpful comments. JK wants to thank J\"org Sch\"urmann and Andrzej Weber for sugesting consideration of Adams operators.

\section{Partitions, Young diagrams and symmetric functions}\label{s:partitions}
Let $n \geq 1$ be an integer. A \emph{partition of $n$ of length $l$} is a non-increasing sequence of positive integers $\lambda = (\lambda_1,\lambda_2, \dots,\lambda_l)$, such that $\sum_{i=1}^l \lambda_i = n$. We write $\lambda \dashv n$ to denote that $\lambda$ is a partition of $n$. We represent each partition by a Young diagram with $n$ boxes and $l$ columns of lengths $\lambda_1,\lambda_2, \dots,\lambda_l$. We think of the boxes in the Young diagram as pairs of natural numbers, and draw the Young diagrams accordingly, e.g.

\begin{center}
	\begin{tabular}{ c c c }
		$\lambda = (3,1)$ & corresponds to  & $\tiny{\yng(1,1,2)}\,.$
	\end{tabular}
\end{center}
The numbering of boxes is such that the bottom left corner is labelled $(0,0)$, the top box in the first column is $(0,2)$ and the unique box in the second column is $(1,0)$. The motivation for this convention is that we want to identify boxes in the Young diagram with monomials in two variables $x,y$, so that the box $(i,j)$ corresponds to $x^iy^j \in \ZZ[x,y]$. 

To each box $\bullet = (i,j)$ in the diagram of $\lambda$ we assign its \emph{arm length} $a_{\lambda,\bullet}$, the number of boxes in the diagram to the right of it, and its \emph{leg length} $b_{\lambda,\bullet}$, the number of boxes above, as depicted below. We call boxes for which $a_{\lambda,\bullet} = b_{\lambda,\bullet}=0$ the \emph{corners} of the diagram. We denote the set of corners of the Young diagram corresponding to $\lambda$ by $C(\lambda)$.

\begin{center}
	\begin{tabular}{ c c }
		\ytableausetup{centertableaux}
		\ytableaushort
		{\none b, \none b ,\none \bullet a , \none}
		* {2,3,3,4} & \begin{tabular}{c}
			$a_{\lambda,\bullet}=1$ \\
			$b_{\lambda,\bullet}=2$
		\end{tabular}
	\end{tabular}
\end{center}

\begin{rem} The arm and leg lengths of a box $\bullet$ are typically denoted $a(\bullet), l(\bullet)$ in the literature, and the partition $\lambda$ is omitted in the notation. We choose to keep $\lambda$, as in many inductive arguments we consider the same box in different diagrams. We also use $b$ for denoting the leg length, so that we can keep $i,j,k,l,m,n$ as the standard choices for denoting natural numbers. 
\end{rem}

For arguments involving inductive procedures with Young diagrams, we often use the following constructions: adding or removing a box and removing the first column of the diagram. We use the following notation.
\begin{adf} \label{df:operation}
	Let $\lambda=(\lambda_1, \lambda_2, \dots, \lambda_l)$ be a nonempty partition. 
	\begin{enumerate}
		\item We denote by $\lambda[1]$ the set of all Young diagrams which are obtained from $\lambda$ by adding one box in such a way that the result is a Young diagram. 
		\item We denote by $\tilde{\lambda}$ the partition obtained from $\lambda$ by removing the first column, i.e. $$\tilde{\lambda}=(\lambda_2,\dots,\lambda_l)\,.$$
	\end{enumerate}
\end{adf}

Let $\Sigma_{n}$ be the permutation group of $n$ elements. When we work with symmetric polynomials in $n+1$ variables, we index the variables from $0$, so we consider elements of the ring $\ZZ[x_0,x_1,\dots, x_{n}]^{\Sigma_{n+1}}$. We denote by $e_i$ the $i$'th elementary symmetric polynomial 
\[e_i(x_0,\dots,x_{n}) = \sum_{0 \leq j_1 < \dots < j_i \leq n} x_{j_1}\dots x_{j_i},\]
and by $p_i$ the $i$'th power-sum polynomial 
\[p_i(x_0,\dots,x_{n}) = \sum_{j=0}^{n} x_j^i\,.\]
In particular
for $i=0$ we have
\[p_0(x_0,\dots,x_{n}) = \sum_{j=0}^{n} x_j^0=\sum_{j=0}^{n} 1=n+1\,.\]
For a partition $\lambda = (\lambda_1, \dots, \lambda_l)$, we set
$$e_{\lambda} := \prod_{i=1}^{l} e_{\lambda_i} \,, \qquad p_{\lambda} := \prod_{i=1}^{l} p_{\lambda_i} \,.$$ 

Finally, we assume that \emph{monomials} have coefficient $1$, so that $x^2y^3$ is a monomial but $3x^2y^3$ in not.

\section{Equivariant K-theory} \label{s:Ktheory}
Let $\TT \simeq (\CC^*)^{\rank \TT}$ be an algebraic torus. We call elements of the free abelian group
$\Hom(\TT,\CC^*)$ characters.
Let $\Rep(\TT)$ denote the representation ring of the torus $\TT$. Every finitely dimensional $\TT$-representation decomposes as a direct sum of one-dimensional representations, therefore
\begin{align} \label{eq:1}
	\Rep(\TT) \simeq \ZZ[\Hom(\TT,\CC^*)]\,.
\end{align}

Let $X$ be a quasiprojective complex $\TT$-variety. The algebraic $\KTh$-theory of $X$, denoted $\KTh(X)$, is
 the Grothendieck group of isomorphism classes of algebraic vector bundles on~$X$. The equivariant algebraic $\KTh$-theory of~$X$, denoted  $\KTh_{\TT}(X)$, is
the Grothendieck group of isomorphism classes of $\TT$-equivariant algebraic vector bundles on $X$. We do not distinguish in the notation between a bundle and its $\KTh$-theory class, when it is clear from the context. 

It follows from \eqref{eq:1} that the equivariant K-theory of a point is isomorphic to the ring of Laurent polynomials
$$\KTh_\TT(pt) \simeq \Rep(\TT) \simeq \ZZ[\Hom(\TT,\CC^*)] \simeq \ZZ[t_1^\pm,\dots,t_{\rank \TT}^\pm]\,,$$
where $t_1,\dots,t_{\rank \TT}$ are the coordinate characters. The equivariant $\KTh$-theory of an arbitrary $\TT$-variety is a $\KTh_\TT(pt)$-module. 
\begin{rem}
	In this paper we also work with torus equivariant cohomology $\coh^{\ast}_\TT$. Let $t_i$ be a coordinate character pf $\TT$ and $\CC_{t_i}$ the corresponding $\TT$-representation. For simplicity we use $t_i$ to also refer to the class of this representation in K-theory and its first Chern class in cohomology. With this convention, we have
	$$ \KTh_\TT(pt) \simeq \ZZ[t_1^\pm,\dots,t_{\rank \TT}^\pm]\,, \qquad \coh^{\ast}_\TT(pt) \simeq \ZZ[t_1,\dots,t_{\rank \TT}]\,.$$
\end{rem}

Let $S\subset \KTh_\TT(pt)$ be the multiplicative system consisting of all the nonzero elements. We denote the localized K-theory ring of $X$ by $S^{-1}\KTh_\TT(X)$. The localization theorem \cite[Theorem 2.1]{Tho} implies that the restriction map induces an isomorphism
$$S^{-1}\KTh_\TT(X) \simeq S^{-1}\KTh_\TT(X^\TT)\,.$$
Let $X$ be a smooth projective $\TT$-variety such that the fixed point set $X^\TT$ is finite.
Then the equivariant K-theory of $X$ is a free $\KTh_\TT(pt)$-module.
In this case we have an inclusion of rings 
$$\KTh_\TT(X) \subset S^{-1}\KTh_\TT(X) \simeq S^{-1}\KTh_\TT(X^\TT)\,.$$
Therefore, restriction to the fixed point set induces the inclusion
$$\KTh_\TT(X)  \mono \KTh_\TT(X^\TT)\simeq \bigoplus_{x\in X^\TT}\ZZ[t_1^\pm,\dots,t_{\rank \TT}^\pm]\,.$$

\begin{rem}\label{rem:BB}
	In the above reasoning one may weaken the assumption that $X$ is projective. It is enough to assume that  there exists a one parameter subgroup of $\TT$ such that the Bia{\l}ynicki-Birula cells (see \cite{B-B1}) cover $X$. The proof is an induction on Bia{\l}ynicki-Birula skeleta, similar to \cite[Section 9]{FRW}.
\end{rem}

For an equivariant complex vector bundle $\mathcal{E}$ over a $\TT$-variety $X$ we define its $\KTh$-theoretic Euler class as
$$\eu(\mathcal{E}):=\sum_{k=0}^{{\rm rk}(\mathcal{E})}(-1)^k[\Lambda^k \mathcal{E}^*] \in \KTh_\TT(X)\,.$$
It is a multiplicative class such that for $\TT$-equivariant line bundles $\L$ we have 
$$\eu(\L):=1-\L^*\,. $$
The Lefschetz-Riemann-Roch formula (LRR for short) allows to compute push-forwards in equivariant $\KTh$-theory using restriction to the fixed point set. Let us recall it here.
\begin{atw}[{\cite[Theorem 3.5]{Tho} and \cite[Theorem 5.11.7]{CG}}] \label{tw:LRR}
	Let $X$ and $Y$ be smooth $\TT$-varieties. Suppose that the fixed point sets $X^\TT$ and $Y^\TT$ are finite. Let $p:X \to Y$ be an equivariant proper morphism.
	For a fixed point $y\in Y^\TT$ and an arbitrary class $\alpha \in S^{-1}K_\TT(X)$ we have
	$$
	\frac{(p_*\alpha)_{|y}}{\eu(T_yY)}=
	\sum_{x\in p^{-1}(y)\cap X^\TT}\frac{\alpha_{|x}}{\eu(T_xX)} \in S^{-1}K_\TT(pt)\,.
	$$
\end{atw}

	 Let $m$ be a natural number and $X$ a $\TT$-variety. The $m_{th}$ Adams operator $\psi^m$ is a map
	 $$
	 \psi^m: \KTh_{\TT}(X) \to \KTh_{\TT}(X)\,,
	 $$
	 which was introduced in the non equivariant setting in \cite{Adams},
	 see \cite{Koeck} for equivariant version and \cite{RR} for a survey. Let us recall its basic properties
	 \begin{pro}
	 	\begin{itemize}
	 		\item  The Adams operator $\psi^m$ is a ring homomorphism.
	 		\item  For an equivariant line bundle $\L \in \Pic^\TT(X)$ we have $$\psi^m(\L)=\L^{m}\,.$$
	 		\item The Adams operator is a natural transformation with respect to pullbacks, i.e. for any $\TT$-equivariant map $f:Y\to X$ we have
	 		$$f^*\circ \psi^m =\psi^m \circ f^*\,.$$
	 	\end{itemize}	  
	 \end{pro}
 	The Adams operators usually do not commute with the pushforwards, yet there is a Riemann-Roch type formula.
 	\begin{adf} [{\cite[Chapter I \textsection 6]{RR}}]
 		The Bott's cannibalistic class $\theta^m$ is a multiplicative characteristic class of vector bundles such that
 		$$\theta^m(\L)=\frac{1-\L^{m}}{1-\L} \in \KTh_\TT(X)\,,$$
 		for any $\TT$-equivariant line bundle $\L$.
 	\end{adf}
 	\begin{atw} [Adams-Riemann-Roch {\cite[Theorem 4.5]{Koeck}, \cite[Theorem 7.6]{RR}}] \label{tw:ARR}
 		Let $X$ and $Y$ be smooth $\TT$-varieties and $\pi:X \to Y$ be an equivariant proper morphism. Denote by $T_\pi=TX-f^*TY$ the virtual tangent bundle to $\pi$.
 		For a class $\alpha \in K_\TT(X)$ we have
 		$$
 		\pi_*\left(\frac{\psi^m(\alpha)}{\theta^m(T^*_\pi)}\right)=\psi^m(\pi_*(\alpha))\,.
 		$$
 	\end{atw}
  	\begin{rem}
 	The above equality holds in the localised equivariant K-theory of $Y$ with rational coefficients. Localisation is necessary to make sense of dividing by the Bott class.
	 \end{rem}

\section{Hilbert scheme of points in a complex plane}\label{prel:hilb}
In this section we recall some standard facts about the Hilbert scheme of points in the plane. For reference, see \cite{MS} or \cite{Nak1}.
\subsection{Definition and torus action}
Let $\Hilb{n} := \Hilb{n}(\Aa_{\CC}^2)$ be the Hilbert scheme of $n$ points in the complex plane. It is a quasi-projective variety parametrising zero-dimensional subschemes of $\Aa_{\CC}^2$ of length $n$. Set-theoretically, $\Hilb{n}$ can be identified with the set of ideals $I \lhd \CC[x,y]$ such that the $\CC$-vector space $\CC[x,y]/I$ has dimension $n$, i.e.,
$$\Hilb{n}:=\{I \lhd \CC[x,y]|\, \dim \CC[x,y]/I=n\}\,.$$
The Hilbert scheme comes with a universal flat family
$$F_n \subset \Hilb{n} \times \CC^2\,,$$
whose fibre over a subscheme $Z\in \Hilb{n}$ is $Z$, treated as a subscheme in $\CC^2$. The push-forward of the sheaf of regular functions on $F_n$ along the projection $p: F_n \to H_n$ defines (a sheaf of sections of) a rank $n$ vector bundle on $\Hilb{n}$
$$\V_n =p_{\ast} \mathcal{O}_{F_n} \in \operatorname{Vect}(\Hilb{n})\,.$$
 We call it the \emph{universal rank $n$ vector bundle on $\Hilb{n}$}. The fibre of the universal bundle $\V_n$ over $I$ is the $n$-dimensional vector space $\CC[x,y]/I$.

Let $\TT=(\CC^{\ast})^2$ be the algebraic torus. The standard $\TT$--action on $\CC^2$ given by
$$(t_1, t_2)\cdot (x,y) := (t_1x, t_2y)$$
induces an action on polynomials $f \in \CC[x,y]$, via
$$(t_1,t_2)\cdot f(x,y):=f(t_1^{-1}x,t_2^{-1}y)\,.$$
It induces a well-defined action on ideals $I \lhd \CC[x,y]$. The fixed points of this action are monomial ideals, i.e. ideals generated by monomials contained in $\Hilb{n}$. To each such an ideal one can assign a Young diagram of a certain partition $\lambda$, in the following way. Let $I$ be a corank $n$ monomial ideal in $\CC[x,y]$. Then there are exactly $n$ monomials that lie outside of $I$ (recall that we assume all monomials to be monic). They constitute a $\CC$-basis of the vector space $\CC[x,y]/I$. The Young diagram corresponding to $I$ consists of boxes $(i,j)$ for which $x^iy^j \notin I$, e.g.

\begin{center}
\begin{tabular}{ c c c }
$I = \langle y^3, xy, x^2 \rangle$ & corresponds to  &\ytableausetup{mathmode, 
 boxsize=1.2em}
\begin{ytableau}
y^2  \\
y \\
1 & x 
\end{ytableau}
\end{tabular}.
\end{center}

\begin{rem} It is commonplace not to distinguish in the notation between the partition $\lambda$ and the associated Young diagram. We extend this convention also to the monomial ideals, so that the monomial ideal corresponding to the partition $\lambda$ is also denoted by $\lambda$ when this does not lead to confusion. When we want to distinguish between the three objects, $\Delta_{\lambda}$ stands for the Young diagram associated to $\lambda$ and $I_{\lambda}$ denotes the monomial ideal associated to $\lambda$.
\end{rem}

 The Hilbert scheme $\Hilb{n}$ is a smooth irreducible quasi-projective variety of dimension~$2n$, \cite{Fogarty}. The tangent space at the point $I \in \Hilb{n}$ can be described as the vector space of $\CC[x,y]$-module homomorphisms from $I$ to the quotient $\CC[x,y]/I$:
\[T_{I}\Hilb{n} = \Hom_{\CC[x,y]}(I, \CC[x,y]/I).\]
At a fixed point $I_{\lambda}$, the torus action induced on the tangent space
$$T_{\lambda}\Hilb{n} := T_{I_{\lambda}} \Hilb{n}$$
equips this $2n$-dimensional vector space with a structure of a representation of $\TT$. Its weights can be read from the corresponding Young diagram. Denote by $q, t$ the coordinate characters of the torus $\TT$.
Then
$$\Rep(\TT) = \ZZ[q^{\pm 1},t^{\pm 1}]\,.$$
Each pair of integers $(k,l)$ determines a representation of $\TT$, given by the character
$$(t_1,t_2) \mapsto t_1^k t_2^l\,,$$
and we call $(k,l)$ the \emph{weight} of this one-dimensional representation. Each box in the diagram represents two weights $(a+1,-b)$ and $(-a,b+1)$, where $a = a_{\lambda}(\bullet), b=b_{\lambda}(\bullet)$ for a box $\bullet$, see \cite[Proposition 5.8]{Nak1} and Example \ref{table-weights-H4}. 

\begin{rem}
Note that with the above convention each box $(i,j)$ in the Young diagram is naturally assigned a weight which is equal to $(i,j)$, since we have an identification between boxes in the diagram and monomials.
\end{rem}

\subsection{Equivariant cohomology and K-theory} \label{s:Eqcoh}
Consider a one-parameter subgroup $\sigma:\CC^*\to \TT$ given by
$$\sigma(t)=(t,t^N)\,,$$
for a sufficiently big integer $N$. The fixed point sets $\Hilb{n}^\TT$ and $\Hilb{n}^\sigma$
 coincide. Moreover, the Bia{\l}ynicki-Birula cells cover the Hilbert scheme, see \cite{ESBBdec} or \cite[Lemma 1.2.4]{YiNing}. Therefore,
for nonequivariant cohomology and $\KTh$-theory we have the following isomorphisms of $\ZZ$-modules
		$$\KTh(\Hilb{n})\simeq \coh^{\ast}(\Hilb{n};\ZZ)\simeq \bigoplus_{\lambda\dashv n} \ZZ\,.$$
Analogous result holds in the equivariant case.  Denote by $q, t$ the coordinate characters of the torus $\TT$.  Then
$$ \KTh_{\TT}(pt) \simeq \ZZ[q^{\pm 1},t^{\pm 1}] \,, \qquad \coh^{\ast}_{\TT}(pt) \simeq\ZZ[q,t]\,. $$
Bia{\l}ynicki-Birula decomposition implies that $\KTh_\TT(\Hilb{n})\simeq \bigoplus_{\lambda\dashv n} \ZZ[q^\pm,t^\pm]$ as $\KTh_{\TT}(pt)$-modules and $ \coh^{\ast}_\TT(\Hilb{n};\ZZ)\simeq \bigoplus_{\lambda\dashv n} \ZZ[q,t]$ as $\coh_{\TT}^{\ast}(pt)$-modules. In particular $\KTh_\TT(\Hilb{n})$ is a free $\KTh_\TT(pt)$-module and $\coh^{\ast}_\TT(\Hilb{n})$ is a free $\coh^{\ast}_\TT(pt)$-module.

Let $i_{\lambda}: \{I_{\lambda}\} \to \Hilb{n}$ be the inclusion of the fixed point $I_{\lambda}$ into the Hilbert scheme. Given a class $\mathcal{E} \in \KTh_{\TT}(\Hilb{n})$, its \emph{restriction to the fixed point $I_{\lambda}$} is its pullback along this inclusion, i.e.
$$i_{\lambda}^{\ast}\mathcal{E} \in \ZZ[q^{\pm 1}, t^{\pm 1}]\,.$$
We denote it by $\mathcal{E}(\lambda)$ or $\mathcal{E}_{|\lambda}$. 
 The localization theorems (see Section \ref{s:Ktheory}) and the Białynicki-Birula theorem (see Remark \ref{rem:BB}) imply the following.
\begin{pro}\label{pro:localization1}
	Restriction to the fixed point variety induces an inclusion of rings
	\begin{align*}
		\KTh_\TT(\Hilb{n}) &\mono \KTh_\TT(\Hilb{n}^\TT) \simeq \bigoplus_{\lambda\dashv n} \ZZ[q^\pm,t^\pm]\,, \\
		\coh^{\ast}_\TT(\Hilb{n}) &\mono  \coh^{\ast}_\TT(\Hilb{n}^\TT) \simeq \bigoplus_{\lambda\dashv n} \ZZ[q,t]\,.
	\end{align*}
\end{pro}
 As a direct consequence, in order to prove equality of two classes in $\KTh_\TT(\Hilb{n})$ or $\coh^{\ast}_\TT(\Hilb{n})$ it is enough to check that their restrictions to each fixed point $I_\lambda$ are equal.

\subsection{Kirwan map} \label{s:Kirwan}
The Kirwan map in the equivariant cohomology is a morphism of $\coh^{\ast}_\TT(pt)$ modules
\begin{align}\label{eq:KirwanH}
	\kappa^{\coh}: \ZZ[q, t] [x_0, \dots, x_{n-1}]^{\Sigma_{n}} \to \coh^{\ast}_{\TT}(\Hilb{n})
\end{align}
 sending the variables $x_i$ to the Chern roots of the tautological bundle $\V_n$. Analogously, the Kirwan map in the equivariant $\KTh$-theory is a morphism of $\KTh_\TT(pt)$-modules
 \begin{align}\label{eq:KirwanK}
 	\kappa^{\KTh}: \ZZ[q^{\pm 1}, t^{\pm 1}][x_0^{\pm 1}, \dots, x_{n-1}^{\pm 1}]^{\Sigma_{n}} \to \KTh_{\TT}(\Hilb{n})\,, \underline{}
 \end{align}
 sending the variables $x_i$ to the $\KTh$-theoretic Chern roots of the tautological bundle $\V_n$.
 \begin{ex}
 	We have
 	$$\kappa^{\KTh} (x_0 + x_1 +\dots + x_{n-1})= \V_n \in \KTh_{\TT}(\Hilb{n})\,.$$
 	More generally, the $k$'th elementary symmetric polynomial is mapped to the class of the $k$'th exterior power of $\V_n$, i.e.
 	$$\kappa^{\KTh}\left( e_k(x_0,\dots,x_{n-1})\right)= \Lambda^k\V_n \in \KTh_{\TT}(\Hilb{n})\,.$$
 \end{ex}
 \begin{adf} \label{df:Powersum}
 	Let $k>0$ be a positive integer.
 	\begin{itemize}
 		\item Let $\PKT{k}{n} \in \KTh_{\TT}(\Hilb{n})$ be the image under $\kappa^{\KTh}$ of the $k$'th power-sum polynomial $p_k(x_0,\dots,x_{n-1})$, i.e.,
 		$$\PKT{k}{n}=\kappa^{\KTh}(x_0^k+x_1^k+\dots+ x_{n-1}^k) \in \KTh_\TT(\Hilb{n}) \,.$$
 		We set $\PKT{0}{n}=n\in \KTh_{\TT}(\Hilb{n})$.
 		\item Let $\Pcoh{k}{n} \in \coh^{2k}_{\TT}(\Hilb{n})$ be the image under $\kappa^{\coh}$ of the $k$'th power-sum polynomial. We set $\Pcoh{0}{n}=n\in \coh^*_{\TT}(\Hilb{n})$.
 		\item For a sequence of non-negative integers $\lambda=(\lambda_1,\dots,\lambda_l)$ we set
 		$$\Pcoh{\lambda}{n}=\prod^{l}_{k=1} \Pcoh{\lambda_k}{n} \in \coh^{*}_{\TT}(\Hilb{n})\,.$$
 		For the empty sequence $\lambda = \varnothing$ let $\Pcoh{\varnothing}{n}=1$.
 		\item We use the same notation $\PKT{k}{n}, \Pcoh{k}{n}$ for the corresponding elements of $\KTh(\Hilb{n})$ and $\coh^*(\Hilb{n})$.
 	\end{itemize}
 \end{adf}
	The nonequivariant Kirwan map is surjective, cf. \cite{ES}. This result can be generalized to equivariant cohomology using graded Nakayama lemma, see Remark \ref{rem:Nakayama} below. It is a folklore knowledge that an analogous result holds in the K-theoretical setting, i.e. the map \eqref{eq:KirwanK} is surjective, see e.g. \cite[Paragraph 3.2]{Smirnov}. As $\Hilb{n}$ is a Nakajima quiver variety, in particular a symplectic reduction, this may be deduced from a much more general fact \cite{QuiverKirwan}. 
	In this paper we do not use the surjectivity of mentioned maps.
	
\begin{rem} \label{rem:Nakayama}
	Surjectivity of the map $\kappa^{\coh}$ may be deduced from the surjectivity of its nonequivariant counterpart.
	Let $R$ be the $\TT$-equivariant cohomology ring of a point, $R_+\subset R$ its ideal of elements in positive degrees, and $M$ the equivariant cohomology of the Hilbert scheme, i.e. 
	$$R=\coh^{\ast}_{\TT}(pt)\simeq\ZZ[q,t]\,, \qquad  R_+= \bigoplus_{k=1}^\infty \coh^{k}_{\TT}(pt)\,, \qquad M=\coh^{\ast}_{\TT}(\Hilb{n})\,.$$
	The image of the Kirwan map $\kappa^{\coh}$ is an $R$-submodule of $M$, denote it by $N$. Surjectivity of the nonequivariant Kirwan map and the Białynicki-Birula decomposition (see Section \ref{s:Eqcoh}) imply that
	$$M=N+R_+M\,.$$
	By the graded Nakayama lemma, it follows that $M=N$, proving the surjectivity of $\kappa^{\coh}$.
\end{rem}

The bundle $\V$ has fibre $\CC[x,y]/I_{\lambda}$ at the point $I_{\lambda}$. Therefore, the restriction of $[\V]$ to $I_{\lambda}$ is equal to the sum of monomials corresponding to the boxes of $\lambda$.
More generally, consider a polynomial
$$W\in \ZZ[q^{\pm 1}, t^{\pm 1}][x_0^{\pm 1}, \dots, x_{n-1}^{\pm 1}]^{\Sigma_n}\,.$$
To compute the restriction $\kappa^{\KTh}(W)_{|\lambda}$ one substitutes the monomials corresponding to boxes in the Young diagram of $\lambda$ as the variables $x_i$.
\begin{ex}
	For $\lambda=(2,1)$, i.e.
	$$\Delta_\lambda = \begin{ytableau}
		y  \\
		1 & x 
	\end{ytableau}$$
	we have 
	\begin{align*}
		\V(\lambda) &= 1 + x + y\,, & \PKT{k}{}(\lambda) &= 1 + x^k + y^k\,, & \Lambda^2\V(\lambda) &= xy + x + y\,. 
	\end{align*}
\end{ex}
For the empty partition $\lambda=\varnothing$ we set
$\V(\varnothing)= \Lambda^k\V(\varnothing) = \PKT{k}{}(\varnothing)=0$.
\\

Adams operators have a very simple description in terms of the Kirwan map. We present it in the following proposition, whose proof is straightforward.
\begin{pro} \label{pro:AdamsKirwan}
	Let $W \in \ZZ[q^{\pm 1}, t^{\pm 1}][x_0^{\pm 1}, \dots, x_{n-1}^{\pm 1}]^{\Sigma_n}$ be a symmetric polynomial. Then
	$$
	\psi^m(\kappa^{\KTh}(W(q,t,x_0, \dots, x_{n-1})))=
	\kappa^{\KTh}(W(q^m,t^m,x^m_0, \dots, x^m_{n-1})).
	$$
	In particular $\psi^m(\PKT{k}{})=\PKT{km}{} \,.$
\end{pro}

\section{Nested Hilbert Scheme}\label{s:nested}
\subsection{Definition and torus action}
\emph{The nested Hilbert scheme} $\Hilb{n,n+i}$ is a scheme parametrizing pairs of subschemes of $\Aa_\CC^2$ contained in one another, i.e.,
\[\Hilb{n,n+i}:=\{ (I,J) : I \supseteq J\} \subset \Hilb{n} \times \Hilb{n+i}\,. \]
It is not smooth unless $i=1$, and for $i=1$, the nested Hilbert scheme $\Hilb{n,n+1}$ is a smooth irreducible variety of dimension $2n+2$, see \cite{Cheah}.

The diagonal action of $\TT$ on the product $\Hilb{n} \times \Hilb{n+1}$ restricts to an action on $\Hilb{n, n+1}$. The fixed points of this action are indexed by pairs of Young diagrams contained in one another, which means that the diagram $\Delta_J$ is obtained from $\Delta_I$ by adding one box. We represent such pairs of diagrams by shading a box in the bigger diagram, representing the added box.

\begin{center}
	\begin{tabular}{ c c c }
		\ytableausetup{smalltableaux}
		\ydiagram[*(white)]
		{1,1}
		*[*(gray)]{0,1+1}& represents the pair &
		$(I,J) = ( \langle y^2, x \rangle, \langle y^2, xy, x^2 \rangle )\,.$
		
	\end{tabular}
\end{center}
Therefore, a fixed point in $\Hilb{n,n+1}^\TT$ is uniquely determined by a partition $\lambda$ of $n+1$ and a choice of a corner box $c\in C(\lambda)$.

\subsection{Restriction to a one-dimensional subtorus} \label{s:subtorus}
Consider a one-dimensional coordinate subtorus of $\TT$:
$$
\Tt=\{(1,t_2)\in \TT\}\,.
$$
A reasoning analogous to the one in Section 1.2 of \cite{Evain2} implies the following description of the fixed point set $\Hilb{n,n+1}^\Tt$.
\begin{pro}
	The fixed point set $\Hilb{n,n+1}^\Tt$ is a disjoint union of affine spaces. Each affine space contains exactly one fixed point of the big torus $\TT$.
\end{pro}
\begin{cor}
	The restriction map
	$$\KTh_\Tt(\Hilb{n,n+1}^\Tt) \xtto{\simeq} \KTh_\Tt(\Hilb{n,n+1}^\TT)$$
	is an isomorphism.
\end{cor}
The Bia{\l}ynicki-Birula decomposition and the localization theorem imply the following result.
\begin{pro}
	The equivariant $\KTh$-theory of the nested Hilbert scheme $\KTh_\Tt(\Hilb{n,n+1})$ is a free $\KTh_\Tt(pt)$-module.
	The restriction to the fixed point set induces an inclusion of rings
	$$\KTh_\Tt(\Hilb{n,n+1}) \mono \KTh_\Tt(\Hilb{n,n+1}^\Tt)\,. $$
\end{pro}
\begin{cor}\label{cor:localization2}
	The composition
	$$
	\KTh_\Tt(\Hilb{n,n+1}) \mono\KTh_\Tt(\Hilb{n,n+1}^\Tt) \xtto{\simeq} \KTh_\Tt(\Hilb{n,n+1}^\TT)
	$$
	is a monomorphism.
\end{cor}
As a direct consequence, in order to prove equality of two classes in $\KTh_\Tt(\Hilb{n,n+1})$ it is enough to check that their restrictions to every fixed point of $\TT$ are equal.

\subsection{Tangent space}
Let $(I,J)$ be a point in $\Hilb{n,n+1}$.
A vector in the tangent space at $(I,J)$ is given by a pair of $\CC[x,y]$-module homomorphisms
$$\varphi_1 \in T_I \Hilb{n}\simeq\Hom_{\CC[x,y]}(I,\CC[x,y]/I)\,,\qquad \varphi_2 \in T_J \Hilb{n+1}\simeq \Hom_{\CC[x,y]}(J,\CC[x,y]/J)\,,$$
compatible with the inclusion $J \subseteq I$, so that the diagram 
\begin{equation} \label{diag:tangent}
	\begin{tikzcd}[row sep=large]
		J & I \\
		{\mathbb{C}[x,y]/J} & {\mathbb{C}[x,y]/I}
		\arrow[hook, from=1-1, to=1-2]
		\arrow[two heads, from=2-1, to=2-2]
		\arrow["{\varphi_1}", from=1-2, to=2-2]
		\arrow["{\varphi_2}", from=1-1, to=2-1]
	\end{tikzcd}
\end{equation}
is commutative. 

The weights of the torus action on the tangent space at a fixed point are described in \cite[Propositions 15 and 16]{Evain} (see also \cite[Lemma 3.1]{FT}).  Here we present an alternative description, better suited to our computations. Let $(I_{\lambda}, J_{\mu})$ be a $\TT$-fixed point in $\Hilb{n, n+1}$, in particular $\mu \in \lambda[1]$. Recall that every box $\bullet$ in $\Delta_{\mu}$ determines two of the $2(n+1)$ tangent weights at $J_{\mu}$ in $\Hilb{n+1}$, i.e., $(a+1,-b)$ and $(-a,b+1)$, where $a = a_{\mu}(\bullet), b=b_{\mu}(\bullet)$. One can identify these weights as minus the slopes of the two arrows depicted below
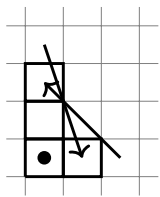
\begin{center}
	\begin{tikzpicture}[line width=1.1pt, scale=0.5]
		\draw[step=1cm,gray,very thin] (-0.5,-0.5) grid (3.5,4.5);
		\draw (0,0) -- (1,0) -- (1,1) -- (0,1) -- (0,0);
		\draw (1,0) -- (2,0) -- (2,1) -- (1,1) -- (1,0);
		\draw (0,1) -- (1,1) -- (1,2) -- (0,2) -- (0,1);
		\draw (0,2) -- (1,2) -- (1,3) -- (0,3) -- (0,2);
		\filldraw [black] (0.5,0.5) circle [radius=4pt];
		\draw [->] (0.5,3.5) -- (1.5,0.5);
		\draw [->] (2.5,0.5) -- (0.5,2.5);
	\end{tikzpicture}  
\end{center}

One arrow starts at the box just above the last box in the column in which the box~$\bullet$ is located and ends at the right-most box in the row of~$\bullet$, while the other arrow starts at the box one to the right of the  right-most box in the row of~$\bullet$ and ends at the top box in the column on~$\bullet$. In particular, every box contributes one weight corresponding to an arrow going south-east (or south) and one arrow going north-west (or west). \\

The tangent weights at the point $(I_{\lambda}, J_{\mu})\in \Hilb{n, n+1}^\TT$ are computed as follows. The diagram $\Delta_{\mu}$ is obtained from $\Delta_{\lambda}$ by adding one box, denote it by
$\tiny{\ydiagram[*(gray)]{1}}$. Let $i_0, j_0$ be the indices of the row and the column in which the added box is located.
Take the set of tangent weights to $\Hilb{n+1}$ at $J_{\mu}$. Recall that each of the weights corresponds to some box in $\Delta_{\mu}$.
For every box below $\tiny{ \ydiagram[*(gray)]{1}}$ replace the weight $(-a, b+1)$ by $(-a,b)$, i.e. shorten the south-east arrow vertically by $1$, and for every box left of $\tiny{\ydiagram[*(gray)]{1}}$ replace the weight $(a+1, -b)$ by $(a,-b)$, i.e. shorten the north-west arrow horizontally by $1$. Keep all the remaining weights intact. For example (the added box is shaded grey) the south-east arrow determined by the box $(0,0)$:

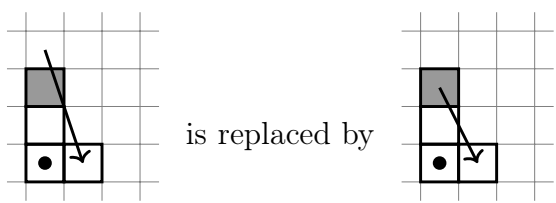
\begin{center}
	\begin{tabular}{ c c c }
		\begin{tikzpicture}[baseline,line width=1.1pt, scale=0.5]
			\draw[step=1cm,gray,very thin] (-0.5,-1.5) grid (3.5,3.5);
			\draw (0,0) rectangle (1,1);
			\draw (1,-1) rectangle (2,0);
			\draw (0,0) rectangle (1,-1);
			\filldraw[fill=black!40!white, draw=black] (0,1) rectangle (1,2);
			\filldraw [black] (0.5,-0.5) circle [radius=4pt];
			\draw [->] (0.5,2.5) -- (1.5,-0.5);
		\end{tikzpicture}  
		& is replaced by &
		\begin{tikzpicture}[baseline,line width=1.1pt, scale=0.5]
			\draw[step=1cm,gray,very thin] (-0.5,-1.5) grid (3.5,3.5);
			\draw (0,0) rectangle (1,1);
			\draw (1,-1) rectangle (2,0);
			\draw (0,0) rectangle (1,-1);
			\filldraw[fill=black!40!white, draw=black] (0,1) rectangle (1,2);
			\filldraw [black] (0.5,-0.5) circle [radius=4pt];
			\draw [->] (0.5,1.5) -- (1.5,-0.5);
		\end{tikzpicture}  
	\end{tabular}
\end{center}

The reasoning behind this procedure is simple. Each weight represented by an arrow is in fact a $\CC[x,y]$-module homomorphism which sends the generator of $J_{\mu}$ at the tail of the arrow to the generator of $\CC[x,y]/J_{\mu}$ at the head of the arrow (see \cite{Haiman2} for details). Unless the box defining the weight is in row $i_0$ or column $j_0$, this is also a well-defined homomorphism from $I_{\lambda}$ to $\CC[x,y]/I_{\lambda}$, making diagram (\ref{diag:tangent}) commutative. For boxes in row $i_0$ or column $j_0$, one of the associated arrows still defines a good tangent vector at $I_{\lambda}$, but the other one does not because it does not determine where the monomial corresponding to the shaded box should be sent - this is the arrow that we modify. See \cite[Section 2.6]{Cheah} for an algebraic computation of the tangent weights which justifies the above description.

\begin{ex} \label{table-weights-H4}
The following table compares the tangent weights at two chosen fixed points in $\Hilb{4}$ and $\Hilb{3,4}$, diagrams with a marked box are in $\Hilb{3,4}$. Each row lists the two weights associated to a chosen box. \\

\begin{center}

\begin{tabular}{|p{0.8cm}|p{2.8cm}|p{2.8cm}|p{2.8cm}|p{2.8cm}|}
\hline 
 & \vspace{0.25em} \ytableausetup{smalltableaux}
\begin{ytableau}
1 \\
 2  \\
3 & 4
\end{ytableau}  \vspace{0.25em} &  \vspace{0.25em} \ytableausetup{smalltableaux}
\begin{ytableau}
*(gray) 1 \\
 2  \\
3 & 4
\end{ytableau}  \vspace{0.25em} & \vspace{0.5em} \ytableausetup{smalltableaux}
\begin{ytableau}
1 & 2 \\
3 & 4
\end{ytableau}  & \vspace{0.5em} \ytableausetup{smalltableaux}
\begin{ytableau}
1 & *(gray) 2 \\
3 & 4
\end{ytableau} \\ \hline  
  {\tiny \young(1)} & \vspace{0.5em}$(0,1), (1,0)$ &\vspace{0.5em} $(0,1), (1,0)$  & \vspace{0.5em} $(-1,1), (2,0)$ &\vspace{0.5em} $(-1,1)$, $\color{red} (1,0)$  \\ \hline
 {\tiny \young(2)} & \vspace{0.5em} $(0,2), (1,-1)$  & \vspace{0.5em} $\color{red}(0,1)$, $(1,-1)$& \vspace{0.5em} $(0,1), (1,0)$ &\vspace{0.5em} $(0,1), (1,0)$  \\ \hline
  {\tiny \young(3)} & \vspace{0.5em} $(-1,3),(2,-2)$ & \vspace{0.5em} $\color{red}(-1,2)$, $(2,-2)$ & \vspace{0.5em} $(-1,2), (2,-1)$ &\vspace{0.5em}$(-1,2), (2,-1)$  \\ \hline
  {\tiny \young(4)} & \vspace{0.5em}$(0,1), (1,0)$  & \vspace{0.5em} $(0,1), (1,0)$ & \vspace{0.5em} $(0,2), (1,-1)$ & \vspace{0.5em} $\color{red} (0,1)$, $(1,-1)$  \\ \hline
\end{tabular}
\end{center}

\end{ex}

\subsection{Tautological line bundle}
Let
$$p: \Hilb{n,n+1} \to \Hilb{n}\,, \qquad \pi: \Hilb{n,n+1} \to \Hilb{n+1}$$
be the restrictions of the projections to the first and the second factor in the product $\Hilb{n} \times \Hilb{n+1}$. Let $\V_n, \V_{n+1}$ be the universal bundles on the respective Hilbert schemes. The pullback $\pi^{\ast} \V_{n+1}$  admits canonical epimorphism onto the bundle $p^{\ast} \V_n$. The kernel of this map is a line bundle on $\Hilb{n, n+1}$. We denote it by
\[\Q_n:=\ker(\pi^{\ast} \V_{n+1} \onto p^{\ast} \V_n)\,.\]
This bundle plays a central role in this paper. Intuitively, it corresponds to "the added point"; its fibre over the point $(I,J) \in \Hilb{n, n+1}$ is equal to $I/J$. Suppose that $\lambda$ is a nonempty partition and $c=(k,l)$ is its corner. Restriction of $\Q_n$ to the  fixed point $(\lambda,c)\in \Hilb{n,n+1}^\TT$ is equal to
$$ (\Q_n)_{|(\lambda,c)}=q^kt^l \in \ZZ[q,t]\,.$$

Graphically, the setup for the rest of the paper can be summarized as follows:
\[\begin{tikzcd}[column sep=small,row sep=large]
	& {\mathbb{C}[x,y]/I} && {\mathbb{C}[x,y]/J} & {I/J} \\
	& {p^{\ast}\V_n} && {\pi^{\ast}\V_{n+1}} & {\Q_n} \\
	{\V_n} && {\Hilb{n,n+1}} && {\V_{n+1}} \\
	& {\Hilb{n}} && {\Hilb{n+1}}
	\arrow["p", from=3-3, to=4-2]
	\arrow["\pi"', from=3-3, to=4-4]
	\arrow[from=3-5, to=4-4]
	\arrow[from=3-1, to=4-2]
	\arrow[from=2-2, to=3-3]
	\arrow[from=2-2, to=3-1]
	\arrow[from=2-4, to=3-3]
	\arrow[from=2-4, to=3-5]
	\arrow["\lrcorner"{anchor=center, pos=0.125, rotate=-45}, draw=none, from=2-2, to=4-2]
	\arrow["\lrcorner"{anchor=center, pos=0.125, rotate=-45}, draw=none, from=2-4, to=4-4]
	\arrow[two heads, from=2-4, to=2-2]
	\arrow[ from=2-5, to=2-4]
	\arrow[two heads, from=1-4, to=1-2]
	\arrow[ from=1-5, to=1-4]
	\arrow[hook, from=1-5, to=2-5]
	\arrow[hook, from=1-4, to=2-4]
	\arrow[hook, from=1-2, to=2-2]
\end{tikzcd}\]

\subsection{Pushforward in the equivariant K-theory}\label{s:pushK}
 The map $\pi:\Hilb{n,n+1} \to \Hilb{n+1}$ is proper. We will be interested in the pushforwards $\pi_*(\Q_n^m) \in \KTh_\TT(\Hilb{n+1})$. We may apply the LRR formula (Theorem \ref{tw:LRR}).  Let us recall that a fixed point in $\Hilb{n,n+1}^\TT$ is uniquely determined by a partition $\lambda$ of $n+1$ and a choice of a corner box $c\in C(\lambda)$.

\begin{pro} \label{cor:LRR1}
	Let \hbox{$\mathcal{E}\in \KTh_{\TT}(\Hilb{n,n+1})$} be an arbitrary class and $\lambda$ be a partition of $n+1$. The push-forward $\pi_{\ast}\mathcal{E}$ can be computed by summing up the local contributions at the fixed points, i.e.
	\begin{equation}\label{eq:push}
		(\pi_{\ast}\mathcal{E})_{|\lambda} = \sum_{c\in C(\lambda)} \mathcal{E}_{|(\lambda,c)} \cdot \frac{\eu(T_{\lambda}\Hilb{n+1})}{\eu(T_{(\lambda,c)}\Hilb{n,n+1})}\,,
	\end{equation} 
\end{pro}

The quotient of the Euler classes in Equation (\ref{eq:push}) is a rational function in variables $q, t$. For a corner $c$ in a partition $\lambda$ we use the notation 
\[r_{\lambda,c}:= \frac{\eu(T_{\lambda}\Hilb{n+1})}{\eu(T_{(\lambda,c)}\Hilb{n,n+1})}\in \ZZ(q,t)\,.\]
Both the numerator and the denominator consist of products of factors of the form $(1-q^{-i}t^{-j})$, where $(i,j)$ are the weights of the action of $\TT$ on tangent spaces to the respective Hilbert schemes. 
For technical reasons it is convenient to consider a rescaled variant of function $r_{\lambda,c}$. Let $\lambda$ be a nonempty partition and $c=(i,j)$ be its corner. We consider
\[\tilde{r}_{\lambda,c} := (\Q_n)_{|(\lambda,c)} \cdot \frac{\eu(T_{\lambda}\Hilb{n+1})}{\eu(T_{(\lambda,c)}\Hilb{n,n+1})}=q^it^j\cdot r_{\lambda,c}\in \ZZ(q,t)\,.\]
A lot of combinatorics of the tangent weights is hidden in the properties of the functions $r_{\lambda,c}$ and $\tilde{r}_{\lambda,c}$. At the end of the paper we include an appendix with all the necessary proofs of the technical properties of these functions.

In the case $\mathcal{E}=\Q_n^m$, Proposition \ref{cor:LRR1} may be restated as follows.

\begin{cor} \label{cor:LRR2}
	Let $\lambda$ be a nonempty partition. We have
	\begin{align*}
		\pi_*[\Q_n^m]_{|\lambda} &= \sum_{c\in C(\lambda)} \left(q^{k(c)}t^{l(c)}\right)^{m} \cdot r_{\lambda,c} 
		= \sum_{c\in C(\lambda)} \left(q^{k(c)}t^{l(c)}\right)^{m-1} \cdot \tilde{r}_{\lambda,c}\,,
	\end{align*}
	where $c=(k(c),l(c))$.
\end{cor}
\begin{pro} \label{pro:duality}
	Let $m$ be an integer. We have
	$$\pi_*[\Q_n^m]=(\pi_*[\Q_n^{1-m}])^*\,.$$
\end{pro}
\begin{proof}
	It is enough to check the formula after restriction to the fixed point set. Let $\tau:\KTh^\TT(pt) \to \KTh^\TT(pt)$ be the map given by $\tau(t)=t^{-1}, \tau(q)=q^{-1}$.
	Due to Corollary \ref{cor:LRR2} it is enough to show that for a partition $\lambda$ we have
	$$
	\sum_{c\in C(\lambda)} \left(q^{k(c)}t^{l(c)}\right)^{1-m} \cdot \tau(\tilde{r}_{\lambda,c})=
	\sum_{c\in C(\lambda)} \left(q^{k(c)}t^{l(c)}\right)^{-m} \cdot \tilde{r}_{\lambda,c}
	$$
	It follows from Corollary \ref{cor:Apinverse}.
\end{proof}

\section{Nakajima operators}\label{s:Nakoperators}

Let
$$\HH:=\bigoplus_{n\geq 0} \coh^{\ast}(\Hilb{n};\QQ)$$
be the sum of the cohomology groups of the Hilbert schemes $\Hilb{n}$ for all $n$, with rational coefficients. The Nakajima--Grojnowski creation and annihilation operators equip the vector space $\HH$ with an action of the Heisenberg algebra, see \cite[Chapter 8]{Nak1}. These operators are defined by intersecting with fundamental classes of certain subvarieties in the product of two Hilbert schemes. Let
\[Q_i^n:= \{ (I,J) : I \supseteq J, \supp(I/J)=\{ x\} \textrm{ for some } x \in \Aa_\CC^2 \} \subset \Hilb{n} \times \Hilb{n+i},\]
i.e $Q_i^n$ consists of pairs of nested subschemes that differ by a subscheme of length $i$ concentrated in one point. Note that for $i=1$ we have $Q_1^n=\Hilb{n,n+1}$. Let $[Q_i^n]$ denote the fundamental class of $Q_i^n$ in the Borel--Moore homology of $\Hilb{n} \times \Hilb{n+i}$ (we need to take the Borel--Moore homology groups because the topological space $\Hilb{n}$ is not compact). Denote the projections in the product $\Hilb{n} \times \Hilb{n+i}$ by $p, q$ respectively, so that $p$ is the projection to $\Hilb{n}$ and $q$ is the projection to $\Hilb{n+i}$. The restriction of the map $q$ to $Q_i$ is proper. The creation operator
$$\qq_i: \coh^{\ast}(\Hilb{n};\QQ) \to \coh^{\ast + 2(i - 1)}(\Hilb{n+i};\QQ)$$
is defined by
\[\qq_i(\alpha) := PD^{-1}\big(q_{\ast}(p^{\ast}\alpha \cap [Q_i^n])\big)\,,\]
where $PD: \coh_{\ast}^{BM}(\Hilb{n+i}) \to \coh^{2n+2i-\ast}(\Hilb{n+i}) $ is Poincar\'{e} duality. The annihilation operators $\qq_{-i}$ are defined by the same subvariety $Q_i^n$ by replacing the roles of $p$ and $q$ (see \cite[Section 8.2]{Nak1} for the discussion of technical issues arising from the fact that $p$ is not proper). For a partition $\lambda=(\lambda_1, \dots, \lambda_l)$ let
$$\qq_{\lambda}:=\qq_{\lambda_l} \circ \dots \circ \qq_{\lambda_1}\,.$$
Let $\mathbb{1}$ be a generator of $\coh^0(\Hilb{0}) \simeq \QQ$. Applying operators $\qq_{\lambda}$ for all partitions $\lambda$ of $n$ to $\mathbb{1}$ one obtains an additive basis of $\coh^{\ast}(\Hilb{n})$. The same construction works for the equivariant cohomology with respect to the $\TT$-action, or for the $\TT$-equivariant Chow ring (see \cite{Evain} for a construction with all the technical details).

\begin{rem}
	In \cite{Nak1}, the subvarieties defining creation operators are constructed in the same way for any smooth projective surface $X$. They lie in $\Hilb{n}(X) \times \Hilb{n+i}(X) \times X$. Each class $\alpha$ in the cohomology of $X$ gives rise to another operator, by pushing down the product of the pullback of $\alpha$ with the class of the correspondence (see \cite[Section 8.3]{Nak1}). For the affine plane the cohomology class $\alpha$ is irrelevant, so we omit it in the construction. In particular, our operator $\qq_i$ is the same as Nakajima's $P_{\mathbb{1}}[i]$.
\end{rem}

Similarly as for cohomology, one can define certain operators on $\KTh$-theory using correspondences (for details of this construction see \cite[Section 5.2.20]{CG}, and for technical issues resulting from the non-projectiveness of $\Hilb{n}$ see \cite[Section 3.2]{SV}). In particular, an element $\mathcal{E} \in \KTh_{\TT}(\Hilb{n}\times \Hilb{n+1})$ with support proper over $\Hilb{n+1}$ defines a map~$\KTh_{\TT}(\Hilb{n}) \to \KTh_{\TT}(\Hilb{n+1})$ given by
\[\mathcal{E}_n \mapsto q_{\ast}(p^{\ast} \mathcal{E}_n \otimes \mathcal{E}).\]
This construction allows us to define the analogue of the cohomological creation operator. For an integer $m$ let $\qq^{\KTh}_{1,m}$ be the operator  corresponding to $\mathcal{E}=i_{\ast}\Q_n^{m}$, where $i$ denotes the inclusion of $\Hilb{n,n+1}$ into the product, so that $\qq^{\KTh}_{1,m}(\mathcal{E}_n) = q_{\ast}(p^{\ast} \mathcal{E}_n \otimes i_{\ast}\Q_n^{m})$.

\begin{rem}
	Operators $\qq^{\KTh}_{1,m}$ are the creation operators in the Ding-Iohara algebra action of \cite{FT}, or the elliptic Hall algebra considered in \cite{SV}. They correspond to operators $e_m$ in the \cite{FT} notation, and {\bfseries f}$_{1,m}$ in the \cite{SV} notation.
\end{rem}

\begin{rem}
	Let
	$p: \Hilb{n,n+1} \to \Hilb{n}$ and $\pi: \Hilb{n,n+1} \to \Hilb{n+1}$ be the standard maps. The projection formula implies that
	$$\qq_1(\alpha) = \pi_{\ast}p^{\ast}(\alpha) \in \coh_\TT^*(\Hilb{n+1})\,, \qquad \qq^{\KTh}_{1,m}(\mathcal{E}) = \pi_{\ast}\big(p^{\ast}\mathcal{E}\otimes\Q_n^m\big) \in \KTh_\TT(\Hilb{n+1})\,.$$

\end{rem}

\section{Pushforward of the tautological bundle}\label{s:Kwyniki}

In this chapter we compute the push-forward $\pi_{\ast}[\Q_n]$ of the $\KTh$-theory class of the universal bundle on $\Hilb{n, n+1}$.
We use notation
$$f(\lambda)=\big(\pi_{\ast}[\Q_n^m]\big)_{|\lambda} \in \ZZ[q^\pm,t^\pm]  \,.$$
The following holds.

\begin{pro} \label{pro:m1}
	For all $n \in \NN$ we have
	$$ \pi_*[\Q_{n}]=[\V_{n+1}] \in \KTh_{\TT}(\H_{n+1})\,. $$
\end{pro}

Let us start by illustrating this theorem with an example.

\begin{ex} Let $n=2$, so that $\pi: \Hilb{2,3} \to \Hilb{3}$. We want to compute $\pi_{\ast}[\Q_2]$. By the Localization Theorem \ref{pro:localization1}, it is enough to understand what happens at each fixed point in $\Hilb{3}$. Let $\lambda = (2,1)$. There are two fixed points in the preimage of $\lambda$ under $\pi$, i.e.
\[\lambda_1 =
\ydiagram[*(gray)]{1}*[*(white)]{0,2}
\,, \qquad
\lambda_2  =
\ydiagram[*(white)]{1,1}*[*(gray)]{0,1+1}
\,.\]
The restriction of $\pi_{\ast}[\Q_2]$ to $\lambda$ is equal to the following sum of local contributions
\[\pi_{\ast}[\Q_2]_{|\lambda} = [\Q_2]_{|\lambda_1} \cdot \frac{\eu(T_{\lambda}\Hilb{3})}{\eu(T_{\lambda_1 }\Hilb{2,3})} + [\Q_2]_{|\lambda_2} \cdot \frac{\eu(T_{\lambda}\Hilb{3})}{\eu(T_{\lambda_2 }\Hilb{2,3})}.\]
The class of $\Q$ restricted to $\lambda_i$ is equal to $q^i t^j$, where $(i,j)$ is the marked box.

Using the descriptions of weights given in Section \ref{prel:hilb} and cancelling the factors which repeat in the numerator and the denominator, this sum reduces to 

\[\pi_{\ast}[\Q_2]_{|\lambda} = t \cdot \frac{1-\frac{q}{t^2}}{1-\frac{q}{t}} + q \cdot \frac{1-\frac{t}{q^2}}{1-\frac{t}{q}} = 1+q+t, \]
which is the sum of the monomials corresponding to the weights of the boxes of $\lambda$. Therefore it is equal to the class of $[\V_3]$  restricted to the fixed point $\lambda$.
\end{ex}

Let us go back to the general case. Since the only two bundles involved are $\Q_n$ and $\V_{n+1}$ we omit subscripts indicating the number of points. We split the proof into two lemmas.
\begin{lemma}\label{lem:limind3}
	Let $\lambda=(\lambda_1,...,\lambda_l)$ be a nonempty partition. Then
	$$\lim_{q\to 0}f(\lambda)=\sum_{i=0}^{\lambda_1-1} t^i \,. $$
\end{lemma}
\begin{proof}
	Denote the corners of $\lambda$ by
	$$c_1 = (k_1,l_1), c_2 = (k_2, l_2), \dots, c_{N} = (k_N, l_N)\,,$$
	ordering them from upper-left to lower-right, so that $l_1, l_2, \dots, l_N$ form a descending sequence.
	Set $l_{N+1}=-1$. Corollaries \ref{cor:LRR2} and \ref{cor:Aplim1} imply that
	\[ \lim_{q\to 0}f(\lambda)=
	\sum_{i=1}^N \lim_{q\to 0} \tilde{r}_{\lambda,c_i}
	=\sum_{i=1}^N  \sum^{l_{i}}_{l_{i-1}+1} t^i=\sum_{i=0}^{l_1} t^i\,. \qedhere \]
\end{proof}

\begin{lemma}\label{lem:limind4}
	Let $\lambda$ be a nonempty partition and $\tilde{\lambda}$ be the partition $\lambda$ without the first column (cf. Definition \ref{df:operation}).  Then the limit
	$$\lim_{q \to \infty} \big(f(\lambda)-qf(\tilde{\lambda}) \big)$$
	exists, i.e., the difference $f(\lambda)-qf(\tilde{\lambda})$ does not contain positive powers of $q$.
\end{lemma}
\begin{proof}
	Denote the corners of $\lambda$ by
	$$c_1 = (k_1,l_1), \dots, c_{N} = (k_N, l_N)\,,$$
	ordering them from upper-left to lower-right, so that $l_1, l_2, \dots, l_N$ form a descending sequence. Suppose that $k_1\neq 0$. Then the corners of $\tilde{\lambda}$ are of the form
	$$c'_1=(k_1-1,l_1),\dots,c'_N=(k_N-1,l_N)\,.$$
	By Corollary \ref{cor:LRR2}
	we have
	$$\lim_{q \to \infty}\big(f(\lambda)-qf(\tilde{\lambda})\big)=
	\sum_{i=1}^N\lim_{q \to \infty} \left(\tilde{r}_{\lambda,c_i}-q\tilde{r}_{\tilde{\lambda},c'_i}\right)\,.
	$$
	This limit exists due to Corollary \ref{cor:Aplim2}.
	
	Suppose now that $k_1=0$.  Then the corners of $\tilde{\lambda}$ are of the form $$c'_2=(k_2-1,l_2),\dots,c'_N=(k_N-1,l_N)\,.$$
	By Corollary \ref{cor:LRR2} we have
	$$\lim_{q \to \infty}f(\lambda)-qf(\tilde{\lambda})=
	\lim_{q \to \infty} \tilde{r}_{\lambda,c_1}+
	\sum_{i=2}^N\lim_{q \to \infty}\left(\tilde{r}_{\lambda,c_i}-q\tilde{r}_{\tilde{\lambda},c'_i}\right)\,.
	$$
	This limit exists due to Corollaries \ref{cor:Aplim2} and \ref{cor:Aplim3}.
\end{proof}

\begin{proof}[Proof of Proposition \ref{pro:m1}]
 Thanks to the Localization Theorem (Proposition \ref{pro:localization1}) we only need to prove that for every fixed point $\lambda$ in $\Hilb{n+1}^\TT$ we have $f(\lambda) = \V(\lambda)$. This is equivalent to the following equality
\[ f(\lambda)
= \sum_{(i,j) \in \Delta_\lambda} q^{i} t^{j} \in \ZZ[q^\pm,t^\pm]\,.\]
The element $f(\lambda)$ is a Laurent polynomial, therefore it can be written as
$$f(\lambda)=a^\lambda_{i,j}\cdot q^it^j \in \ZZ[q^\pm,t^\pm]\,$$
 for some $a^\lambda_{i,j}\in \ZZ$. We need to prove that 
\begin{align} \label{eq:f1}
	a^\lambda_{i,j}= 
	\begin{cases}
		1 \text{ if } (i,j) \in \Delta_\lambda\,, \\
		0\text{ otherwise \,.}
	\end{cases}
\end{align} 
We use induction on the sum of $\lambda$. For the empty partition the claim is obvious. 

Let us focus on the inductive step. Let $\tilde{\lambda}$ be the partition $\lambda$ without the first column, cf. Definition \ref{df:operation}. By the inductive assumption the thesis holds for $\tilde{\lambda}$. Lemma \ref{lem:limind3} implies that when $i\le 0$, Equation \eqref{eq:f1} holds.
For $i\ge 1$, Lemma \ref{lem:limind4} implies that
$$
	a^\lambda_{i,j}=a^{\tilde{\lambda}}_{i-1,j}\,.
$$
Therefore, for $i\ge 1$ Equation \eqref{eq:f1} follows from the inductive assumption.  
\end{proof}

Propositions \ref{pro:m1} and \ref{pro:duality} imply the following fact.
\begin{cor} \label{cor:m0}
	For all $n \in \NN$ we have
	$$ \pi_*[\O_{\Hilb{n,n+1}}]=[\V^*_{n+1}] \in \KTh_{\TT}(\H_{n+1})\,. $$
\end{cor}

\section{Adams-Riemann-Roch} \label{s:Adams}
Let $\Tt \subset \TT$ be a one-dimensional subtorus acting only on the second variable, cf. Section \ref{s:subtorus}.
In this section, we generalize Proposition \ref{pro:m1} in the $\Tt$-equivariant $\KTh$-theory using Adams-Riemann-Roch theorem.
\begin{atw} \label{tw:push}
 For an integer $m$ we have 
	$$\pi_*[\Q^{m}]=
	\frac{1-t^{-m}}{1-t^{-1}}\cdot \PKT{m}{}
	-\frac{1-t^{-(m-1)}}{1-t^{-1}}\cdot\frac{1}{t}\cdot\PKT{m-1}{}\in \KTh_{\Tt}(\Hilb{n+1})\,,$$
 where $\PKT{i}{}$  denotes the image of the power-sum polynomial under the Kirwan map, see Definition \ref{df:Powersum}.
\end{atw}

As a consequence, we obtain the pushforward formula in the non-equivariant $\KTh$-theory.
\begin{cor}\label{cor:pushnoneq}
 For an arbitrary integer $m$ we have
	$$\pi_*[\Q^m]=m\cdot\PKT{m}{} -(m-1)\cdot\PKT{m-1}{}\in \KTh(\Hilb{n+1})\,.$$
\end{cor}

The rest of this section is devoted to the proof of Theorem \ref{tw:push}.
	First, we prove two lemmas.
\begin{lemma}\label{lem:Adams1}
	Let $cl^*$ be an arbitrary multiplicative class of $\Tt$-vector bundles. We have
	\begin{align} \label{w:Tpi}
		cl^*(T^*_\pi)=\frac{cl^*(t^{-1})}{cl^*((t\Q)^{-1})} \in K^\Tt(\Hilb{n,n+1}),
	\end{align}
	where $T^*_\pi$ is the relative cotangent bundle to the projection $\pi: \Hilb{n,n+1} \to \Hilb{n+1}$.
\end{lemma}
\begin{proof}
	It is enough to prove the formula after restriction to the fixed point set $\Hilb{n,n+1}^\TT$ (see Corrolary \ref{cor:localization2}).
	Let $\lambda$ be a partition and $(k,l)$ its corner. Let $cl^*$ be an arbitrary characteristic class. The description of tangent weights (see Section \ref{s:nested}) 
	implies that
	$$ (T_{\pi})_{|\lambda,(k,l)}= \CC_{t}+\CC_{2t}+\dots+ \CC_{kt}-\CC_{2t}-\CC_{3t}-\dots- \CC_{(k+1)t} =\CC_{t}-\CC_{(k+1)t} $$
	as a virtual $\Tt$--representation. Therefore, the restriction of the left hand side of equation \eqref{w:Tpi} is equal to
	$$cl^*(T^*_\pi)_{|\lambda,(k,l)}=\frac{cl^*(t^{-1})}{cl^*(t^{-k-1})}$$
	The restriction of the line bundle $\Q$ to the fixed point $\lambda,(k,l)$ is equal to $kt$, thus the  restrictions of both sides of Equation \eqref{w:Tpi} are equal.
\end{proof}
\begin{lemma}\label{lem:Adams2}
	Let $\L$ be an arbitrary $\Tt$-equivariant line bundle. For an integer $m\ge 2$ we have
	$$
	\psi^m(\L)=\psi^m(\L)\cdot\theta^m\left((t\L)^{-1}\right) - t^{-1}\cdot\psi^{m-1}(\L)\cdot\theta^{m-1}\left((t\L)^{-1}\right)\,.
	$$
\end{lemma}
\begin{proof}
	It is direct computation applying
	$$\psi^m(\L)=\L^m\,, \qquad \theta^m\left((t\L)^{-1}\right)=\frac{1-(t\L)^{-m}}{1-(t\L)^{-1}}\,.$$
\end{proof}
\begin{proof}[Proof of Theorem \ref{tw:push}]
	Due to Proposition \ref{pro:duality}
	it is enough to prove the theorem for $m \ge 1$. For $m=1$ the theorem follows from Proposition \ref{pro:m1}, so we may assume that $m \ge 2.$ \\
	We have $\Q^m=\psi^m(\Q)$, because $\Q$ is a line bundle. Lemma \ref{lem:Adams2}
	for $\L=\Q$ implies that
	$$
	\Q^m=\psi^m(\Q)\cdot\theta^m\left((t\Q)^{-1}\right) -
	\frac{1}{t}\cdot\psi^{m-1}(\Q)\cdot \theta^{m-1}\left((t\Q)^{-1}\right)\,.
	$$
	We rewrite this formula  using Lemma \ref{lem:Adams1} for $cl^*$ equal to the Bott classes $\theta^m$ and $\theta^{m-1}$ to obtain 
	$$
	\Q^m=
	\theta^m(t^{-1})\cdot\frac{\psi^m(\Q)}{\theta^m(T^*_\pi)}-\frac{\theta^{m-1}(t^{-1})}{t}\cdot
	\frac{\psi^{m-1}(\Q)}{\theta^{m-1}(T^*_\pi)}\,.
	$$
	Applying the pushforward $\pi^*$  to this equation and using Adams-Riemann-Roch (Theorem \ref{tw:ARR})
	we get 
	$$
	\pi_*[\Q^m]=
	\theta^m(t^{-1})\cdot\psi^m(\pi_*[\Q])-\frac{\theta^{m-1}(t^{-1})}{t}\cdot\psi^{m-1}(\pi_*[\Q])\,.
	$$
	The theorem follows from the fact that $\pi_*[\Q]=\PKT{1}{}$ (Proposition \ref{pro:m1})
	and $\psi^m(\PKT{1}{})=\PKT{m}{}$ (Proposition \ref{pro:AdamsKirwan}).
\end{proof}
\begin{rem}
	We proved the pushforward formula for the one dimensional torus $\Tt$. The restriction to this torus is crucial for lemma \ref{lem:Adams1}. The weights at the fixed points simplify and the formula for the Bott class becomes computable. At the same time the fixed point locus is still relatively easy (it is a disjoint union of affine spaces), so the localization theorem may still be used effectively.
\end{rem}
\begin{rem}
	In the $\TT$-equivariant K-theory one can prove that
	$$\pi_*[\Q^{2}]=
	\frac{1-(qt)^{-2}}{1-(qt)^{-1}}\cdot \PKT{2}{}
	-\frac{1}{qt}\cdot\PKT{1}{} +(1-q^{-1})(1-t^{-1})[\Lambda^2\V]\,,$$
	yet this formula does not generalize to higher powers of $\Q$ in an easy way.
\end{rem}

\section{Transition to cohomology}\label{s:Hwyniki}
In this section we transfer our results to cohomology. Our aim is to prove the following theorem.
\begin{atw} \label{tw:pushH} 
	In $\Tt$-equivariant cohomology of the Hilbert scheme $\coh_{\Tt}^{2m}(\Hilb{n+1})$, we have
	$$\pi_*\left(c_1(\Q_n)^m\right)=\sum_{k=0}^{m}a_{k,m}\cdot t^{m-k}\cdot \Pcoh{k}{n+1} \,, $$
	where the $a_{k,m}$ are the coefficients of the polynomial
	$$\sum_{k=0}^{m}a_{k,m}\cdot x^k=x^m(x+1)-(x-1)^mx\,. $$
\end{atw}
\begin{cor} \label{cor:pushH}
	In non-equivariant cohomology, we have
	$$\pi_*\left(c_1(\Q_n)^m\right)=(m+1)\cdot \Pcoh{m}{n+1} \in \coh^{2m}(\Hilb{n+1})\,. $$
\end{cor}
To prove the above theorem we use the equivariant Grothendieck--Riemann--Roch theorem \cite{EGGRR}.
Let $\ch^\Tt$
be the $\Tt$-equivariant Chern character and $\td^\Tt$ the $\Tt$-equivariant Todd class. The Todd class is a multiplicative characteristic class corresponding to the power series
$$\frac{x}{1-e^{-x}}=1+\frac{x}{2}+\cdots \,.$$
The Grothendieck--Riemann--Roch theorem implies that for
an element $\mathcal{E}\in \KTh_{\Tt}(\Hilb{n,n+1})$, we have
\begin{align} \label{w:GRR}
	\ch^{\Tt}(\pi_*\mathcal{E})=\pi_*(\ch^{\Tt}(\mathcal{E})\cdot \td^{\Tt}(T_\pi))\in \coh^*_{\Tt}(\Hilb{n+1})\,,
\end{align}
where $T_\pi$ is the relative tangent bundle to the projection $\pi$, i.e.
$$\td^{\Tt}(T_\pi)=\frac{\td^{\Tt}(T\Hilb{n,n+1})}{\td^{\Tt}(\pi^*T\Hilb{n+1})}\,.$$
The Todd class corresponds to a power series starting with $1$, therefore
\begin{align}\label{w:td}
	\td^{\Tt}(T_\pi)=1+\coh_\Tt^{> 0}(\Hilb{n,n+1})\,.
\end{align}

We split the proof of Theorem \ref{tw:pushH} into several lemmas.
\begin{lemma} \label{lem:ch1}
	Let $\mathcal{L}$ be a $\Tt$-equivariant line bundle on $\Hilb{n,n+1}$. Consider the class
	$$\ch^{\Tt}\big(\pi_*\big((\mathcal{L}-1)^m\big)\big)\in \coh^*_\Tt(\Hilb{n+1})\,.$$
	Its homogeneous part of degree $2m$ is equal to $\pi_*\left(c_1(\mathcal{L})^m\right)$. 
\end{lemma}
\begin{proof}
	Formula \eqref{w:GRR} implies that
	$$\ch^{\Tt}\big(\pi_*\big((\mathcal{L}-1)^m\big)\big)=\pi_*\big(\ch^{\Tt}\big((\mathcal{L}-1)^m\big)\cdot \td^{\Tt}(T_\pi)\big)\,. $$
	The Chern character is multiplicative and additive, therefore
	$$\ch^{\Tt}\big((\mathcal{L}-1)^m\big)=\left(e^{c_1(\mathcal{L}_n)}-1\right)^m=
	\begin{cases}
		0  \text{ in degrees } 0,1,\dots,2m-1 \\
		c_1(\mathcal{L})^m \text{ in degree } 2m.
	\end{cases} $$
	Equation \eqref{w:td} implies that the same formula is true after multiplication with the Todd class, i.e.
	$$\ch^{\Tt}\big((\mathcal{L}-1)^m\big)\cdot \td^{\Tt}(T_\pi)=
	\begin{cases}
		0  \text{ in degrees } 0,1,\dots,2m-1 \\
		c_1(\mathcal{L})^m \text{ in degree } 2m.
	\end{cases} $$
	The lemma follows from the fact that the pushforward $\pi_*$ preserves the grading.
\end{proof}

\begin{lemma} \label{lem:ch2}
	Let $m$ be a non-negative integer. There exist rational numbers  \hbox{$a_{0,m},\dots,a_{m,m}$} such that for an arbitrary $n\in \NN$ we have 
	$$\pi_*\big(c_1(\Q_n)^m\big)=
	\sum_{k=0}^{m}a_{k,m} \cdot t^{m-k}\cdot\Pcoh{k}{n+1}\in \coh_{\Tt}^{2m}(\Hilb{n+1})\,. $$
\end{lemma}
\begin{proof}
	Theorem \ref{tw:push} and Corollary \ref{cor:m0} imply that there exist polynomials
	$A_{s,m}\in \ZZ[x]$
	such that, for an arbitrary $n\in \NN$, we have 
	$$\pi_*\big((\Q_n-1)^m\big)=
	\sum_{s=-1}^m A_{s,m}(t^{-1})\cdot \PKT{s}{n+1} \in \KTh_{\Tt}(\Hilb{n+1})\,.$$
	The Chern character satisfies
	$$\ch^{\Tt}(A_{s,m}(t^{-1}))=A_{s,m}(e^{-t})\,, \qquad\ch^{\Tt}(\PKT{s}{n+1})=e^{s\Pcoh{1}{n+1}}=
	\sum_{k=0}^\infty \frac{s^k\cdot \Pcoh{k}{n+1}}{k!}\,.$$
	Therefore, we have the following equality in $\coh^*_{\Tt}(\Hilb{n+1})$.
	\begin{align*}
		\ch^{\Tt}\big(\pi_*\big((\Q_n-1)^m\big)\big)&=
		\sum_{s=-1}^m \ch^{\Tt}(A_{s,m}(t^{-1}))\cdot \ch^{\Tt}(\PKT{s}{n+1}) \\
		&=\sum_{s=-1}^m\left(A_{s,m}(e^{-t})\cdot \sum_{k=0}^\infty \frac{s^k\cdot \Pcoh{k}{n+1}}{k!}\right)\\
		&=\sum^\infty_{k=0}\left(\Pcoh{k}{n+1}\cdot\sum_{s=-1}^{m}\frac{s^k\cdot A_{s,m}(e^{-t})}{k!} \right)\,.
	\end{align*}
	The expression
	$$B_{k,m}=\sum_{s=-1}^{m}\frac{s^k\cdot A_{s,m}(e^{-x})}{k!} \in \QQ[[x]] $$
	is a power series.
	Thanks to Lemma \ref{lem:ch1} the number $a_{k,m}$ is the coefficient of $B_{k,m}$ corresponding to $x^{m-k}$.
\end{proof} 
\begin{rem}
	The polynomials $A_{s,m}$ from the proof of Lemma \ref{lem:ch2} can be computed explicitly. For $s \ge 1$ we have
	$$A_{s,m}(x)=\binom{m}{s}\cdot(-1)^s\cdot(1+x+\dots+x^s)-\binom{m}{s+1}\cdot(-1)^{s+1} \cdot(x+\dots+x^s) \,.$$
	For $s\in\{-1, 0\}$ we have $A_{0,m}(x)=0$ and $A_{-1,m}(x)=(-1)^m$.
\end{rem}
\begin{lemma} \label{lem:LRRcoh}
	Let $\lambda=(n+1)$ be the partition of $n+1$ consisting of a single summand. In the $\TT$-equivariant cohomology we have 
	$$\pi_*\big(c_1(\Q_n)^m\big)_{|\lambda}=n^m(n+1)\cdot t^m \in \ZZ[t,q]\,. $$
	The same formula is true in the $\Tt$-equivariant cohomology.
\end{lemma}
\begin{proof}
	The Young diagram of the partition $\lambda$ has only one corner $c$. The Lefschetz--Riemann--Roch formula in cohomology \cite{AB,BV} implies that
	$$
	\pi_*\left(c_1(\Q_n)^m\right)_{|\lambda}=
	\frac{\eu^{\coh}(T_{\lambda}\Hilb{n+1})}
	{\eu^{\coh}(T_{\lambda,c}\Hilb{n+1})}\cdot c_1(\Q_n)_{|\lambda}^m
	=(n+1)\cdot(nt)^m=n^m(n+1)\cdot t^m\,.
	$$
	Here $\eu^{\coh}$ denotes the equivariant cohomological Euler class. The result in $\Tt$-equivariant case follows from the formula in $\TT$-equivariant cohomology.
\end{proof}

\begin{proof}[Proof of Theorem \ref{tw:pushH}]
	Fix a positive integer $m$. Consider a polynomial
	$$W_m(x)= \sum_{k=0}^{m}a_{k,m}x^k \in\QQ[x]$$
	where $a_{k,m}$ are the rational numbers from Lemma \ref{lem:ch2}. We need to prove that
	\begin{align} \label{eq:polynomial}
		W_m(x)= x^m(x+1)-(x-1)^mx\,.
	\end{align}
	Let $\lambda=(n+1)$ be the partition of $n+1$ consisting of a single summand. It corresponds to  the vertical Young diagram. We have
	\begin{align*}
		\Big(\sum_{k=0}^{m}a_{k,m}\cdot t^{m-k}\cdot \Pcoh{k}{n+1}\Big)_{|\lambda}&=
		\sum_{k=0}^{m}a_{k,m}\cdot t^{m-k}\cdot p_{k}(0,t,2t,\dots,nt) \\
		&=\left(W_m(0)+W_m(1)+\dots+W_m(n)\right)\cdot t^m \,. 
	\end{align*}
	Lemmas \ref{lem:ch2} and \ref{lem:LRRcoh} provide an alternative way to compute this element. It follows that for all $n \in \NN$ we have
	$$W_m(0)+W_m(1)+\dots+W_m(n)=n^m(n+1)\,. $$
	Therefore for all $n\in \NN$ we have
	$$W_m(n)=n^m(n+1)-(n-1)^mn \,,$$
	which implies formula \eqref{eq:polynomial}.
\end{proof}
\section{Formula for Nakajima's creation operators}\label{s:Nakwyniki}
The main goal of this paper is to answer the question of Nakajima \cite[Question 9.6]{Nak1}, i.e. to compute classes
$$\qq_i(c_k(\V_n)) \in \coh^{2k+2(i-1)}(\Hilb{n+i})\ $$
in terms of the Kirwan map. We also consider a variant of this question in equivariant cohomology.
\subsection{Operator $\qq_1$}
Let us recall that we consider the projection maps
$$p: \Hilb{n,n+1} \to \Hilb{n}\,, \qquad \pi: \Hilb{n,n+1} \to \Hilb{n+1}\,.$$

\begin{pro} \label{pro:q1}
	We have the following equality in equivariant K-theory $\KTh_{\TT}(\Hilb{n,n+1})$.
	\begin{align*}
		[p^*\V_n]=[\pi^*\V_{n+1}]-[\Q_n] \,.
	\end{align*}
 We have the following equalities in equivariant cohomology $\coh^*_{\TT}(\Hilb{n,n+1})$.
	\begin{align*}
		p^*c_\bullet(\V_n)&=\pi^*\big(c_\bullet(\V_{n+1})\big)\cdot (1+c_1(\Q_n))^{-1} \,,\\
		p^*\Pcoh{k}{n}&=\pi^*\Pcoh{k}{n+1} -c_1(\Q_n)^{k}	 \,,
	\end{align*}
	where $c_\bullet(-)$ denotes the full Chern class.
\end{pro}
\begin{proof}
	All equations follow from the short exact sequence
	\[0\to \Q_n \to \pi^{\ast}\V_{n+1} \to p^{\ast}\V_n  \to 0 \qedhere \]
\end{proof}

The above proposition together with the projection formula give formulas for the operator $\qq_1$.
\begin{cor} \label{cor:q1}
	Let $k\ge 0$. The following holds in the equivariant cohomology $\coh_\TT^{2k}(\Hilb{n+1})$.
	\begin{align}
		\nonumber\qq_1(c_k(\V_n))=&\sum_{m=0}^k (-1)^m\cdot c_{k-m}(\V_{n+1})\cdot\pi_*(c_1^m(\Q_n)) \,, \\
		\label{eq:powersum}\qq_1(\Pcoh{k}{n})=&\Pcoh{k}{n+1}\cdot\pi_*(1)-\pi_*(c_1^k(\Q_n))\,.
	\end{align}
\end{cor}
\begin{rem}
	The maps forgetting the torus action
	$$\coh_\TT^{\ast}(\Hilb{n+1})\to \coh_\Tt^{\ast}(\Hilb{n+1})\qquad \coh_\TT^{\ast}(\Hilb{n+1}) \to \coh^{\ast}(\Hilb{n+1}) $$
	commute with Kirwan map, as well as with pushforwards and pullbacks. Therefore, the  equalities from Corollary \ref{cor:q1} hold also in $\coh_\Tt^{*}(\Hilb{n+1})$ and $\coh^{*}(\Hilb{n+1})$.
\end{rem}
Corollary \ref{cor:pushH} implies the following result.
\begin{pro} \label{tw:q1}
	Let $k\ge 0$. The following holds in nonequivariant cohomology $\coh^{2k}(\Hilb{n+1})$.
	$$\qq_1(c_k(\V_n))=\sum_{m=0}^k (-1)^m(m+1)\cdot c_{k-m}(\V_{n+1})\cdot \Pcoh{m}{n+1}\,.$$
\end{pro}
\begin{rem}
	An analogous formula may be obtained in the equivariant cohomology $\coh_\Tt^{2k}(\Hilb{n+1})$ by using Theorem \ref{tw:pushH} instead of Corollary \ref{cor:pushH}. It is slightly more complicated:
	$$
	\qq_1(c_k(\V_n))=\sum_{m=0}^k \sum_{u=0}^m (-1)^m a_{u,m} \cdot c_{k-m}(\V_{n+1})  \cdot \Pcoh{u}{n+1} \cdot t^{m-u} \in \coh^{2k}_{\Tt}(\Hilb{n+1})\,,
	$$
	where $a_{u,m}$ are the coefficients from Theorem \ref{tw:pushH}, i.e.
	$$ \sum_{u=0}^{m}a_{u,m}\cdot x^u=x^m(x+1)-(x-1)^mx \,.$$
\end{rem}
Chern classes are images of elementary symmetric polynomials under the Kirwan map. It turns out that the power sum basis is better suited for our computations. For a sequence $\lambda=(\lambda_1,\dots,\lambda_l)$ and a subset $A\subset\{1,\dots,l\}$, let $\lambda_A$ be the sequence obtained by removing indices corresponding to elements of $A$. We use the notation
$$l(A)=\sum_{i\in A}\lambda_i\,. $$
For the empty subset, we let $l(\varnothing)=0$. Let us recall that $\Pcoh{0}{n}=n$ and $\Pcoh{\varnothing}{n}=1$.
\begin{pro} \label{tw:NakP}
	Let $k\ge 0$. The following holds in the nonequivariant cohomology $\coh^{2k}(\Hilb{n+1})$.
	$$\qq_1(\Pcoh{k}{n})=(n-k)\cdot \Pcoh{k}{n+1}\,.$$
	More generally, let $\lambda=(\lambda_1,\dots,\lambda_l)$ be a sequence of non-negative integers. Then
	$$\qq_1(\Pcoh{\lambda}{n})=\sum_{A\subseteq\{1,\dots,l\}}(-1)^{|A|}(l(A)+1)\cdot\Pcoh{{\lambda_A}}{n+1}\cdot\Pcoh{l(A)}{n+1}\,.$$
\end{pro}
\begin{proof}
	The first part follows from Corollaries \ref{cor:q1} and \ref{cor:pushH}. The second part uses the fact that
	$$p^*\Pcoh{\lambda}{n}=\sum_{A\subseteq\{1,\dots,l\}}(-1)^{|A|}\cdot\pi^*\Pcoh{\lambda_A}{n+1}\cdot c_1(\Q_n)^{l(A)} \in \coh^*(\Hilb{n,n+1})\,,$$
	which is a consequence of Proposition \ref{pro:q1}.
\end{proof}
Analogously, one may describe the action of $\qq^{\KTh}_{1,m}$ operator.
\begin{pro} \label{pro:qK}
	The following holds in the nonequivariant K-theory $\KTh(\Hilb{n+1})$.
	$$\qq^K_{1,0}(\PKT{k}{n})=	\PKT{k}{n+1}\cdot \PKT{-1}{n+1}-k\PKT{k}{n+1}+(k-1)\PKT{k-1}{n+1}\,,$$
	More generally
	$$\qq^K_{1,m}(\PKT{k}{n})=\PKT{k}{n+1}\cdot\big(m\PKT{m}{n+1}-(m-1)\PKT{m-1}{n+1}\big)-(k+m)\PKT{k+m}{n+1}+(k+m-1)\PKT{k+m-1}{n+1}\,.$$
\end{pro}
\begin{proof}
	We have
	\begin{align*}
		\qq^K_{1,m}(\PKT{k}{n})&=\pi_*\big((\pi^*\PKT{k}{n+1}-\Q_n^{k})\cdot \Q_n^m\big)\\
		&=\PKT{k}{n+1}\pi_*(\Q_n^m)-\pi_*(\Q_n^{k+m})\,.
	\end{align*}

	The result follows from Corollary \ref{cor:pushnoneq}.
\end{proof}
\begin{rem}
	An analogous formula in the equivariant K-theory $\KTh_\Tt(\Hilb{n+1})$ can be stated using Theorem \ref{tw:push} instead of Corollary \ref{cor:pushnoneq}.
\end{rem}
\begin{rem} \label{rem:Lehn}
	The relation between the characteristic classes of the tautological bundle and the Nakajima's operators is studied in \cite{Lehn}. There an arbitrary smooth irreducible surface~$X$ is considered. The author  considers operators on $\coh^{\ast}(\Hilb{n}(X))$ associated  with vector bundles on $X$ -- every bundle on $X$ canonically determines a vector bundle on $\Hilb{n}(X)$, which acts on the cohomology of $\Hilb{n}(X)$ by multiplication with its total Chern class.
	For the affine plane he studies the action of the tautological bundle on the symmetric functions, using Nakajima--Grojnowski identification
	$$\QQ[p_1,p_2,...] \simeq \bigoplus_{n\in\NN}\coh^{\ast}(\Hilb{n};\QQ)\,.$$
	The resulting operator is described in \cite[Theorem 4.10]{Lehn} and \cite[Theorem 4.1]{Lehn2} using different methods, see also \cite{Boissiere3} for a generalization to K-theory using McKay correspondence \cite{BKR}. In our paper we solve a dual problem. Under the Kirwan map the multiplication by the characteristic classes of the universal bundle are easy to describe and the Nakajima operators are complicated. Our equivariant formulas (e.g. Theorem~\ref{tw:pushH}) are independent from previous results.  
\end{rem}
\subsection{Higher operators}
The formulas for higher Nakajima's operators in the nonequivariant cohomology or $\Tt$-equivariant cohomology may be deduced from formulas for $\qq_1$. \\
The auxiliary operator
$$\rho:\coh_\TT^*(\Hilb{n})\to \coh_\TT^{*+1}(\Hilb{n+1})$$
is defined in \cite[Definition 33]{Evain}. Thanks to \cite[Corollary 30 and Theorem 34]{Evain} it satisfies
$$\rho(\mathcal{E})=(-1)\cdot\pi_*(c_1(\Q_n)\cdot p^*\mathcal{E})\,. $$
Therefore, a reasoning analogous to the one in the proof of Proposition \ref{tw:NakP} implies the following proposition.
\begin{pro} \label{pro:rho}
	In the nonequivariant cohomology $\coh^{\ast}(\Hilb{n+1})$ the following holds.
	$$\rho(\Pcoh{k}{n})=(k+2)\cdot \Pcoh{k+1}{n+1}-2\cdot \Pcoh{k}{n+1}\cdot\Pcoh{1}{n+1}\,.$$
	More generally, let $\lambda=(\lambda_1,\dots,\lambda_l)$ be a sequence of non-negative integers. Then
	$$\rho(\Pcoh{\lambda}{n})=\sum_{A\subseteq\{1,\dots,l\}}(-1)^{|A|+1}(l(A)+2)\cdot\Pcoh{{\lambda_A}}{n+1}\cdot\Pcoh{l(A)+1}{n+1}\,.$$
\end{pro}
To compute the higher operators we use the following inductive result.
\begin{atw}[{\cite[Proposition 3.12]{Lehn}, \cite[Theorem 34]{Evain}}] \label{tw:Evain} 
	For $i \ge 2$, we have
	$$\qq_i=\frac{\rho\circ \qq_{i-1}- \qq_{i-1}\circ \rho}{i-1} \,.$$
	This result is valid also in the equivariant setting.
\end{atw}

Our formulas (Propositions \ref{tw:NakP} and \ref{pro:rho}) allow for an inductive computation of the Nakajima's operators in terms of the Kirwan map. They yield a formula for the image of an arbitrary symmetric polynomial written in the power-sum basis.

\begin{atw} \label{tw:qm}
	Let $\lambda=(\lambda_1,\dots,\lambda_l)$ be a sequence of nonnegative integers and $m$ a positive integer. Then
	\begin{align*}
		\qq_m(\Pcoh{\lambda}{n})=(-1)^{m+1}\cdot\sum_{A\subseteq\{1,\dots,l\}}(-1)^{|A|}m^{|A|}\cdot(l(A)+m)\cdot\Pcoh{{\lambda_A}}{n+m}\cdot\Pcoh{l(A)+m-1}{n+m}\,.
	\end{align*}
\end{atw}
\begin{cor} \label{cor:qm}
	let $k\ge 0$ be a nonegative integer and $m>0$ a positive integer. Then
	\begin{align*}
		\qq_m(\Pcoh{k}{n})=(-1)^{m+1}\cdot\left(m\cdot\Pcoh{k}{n+m}\cdot\Pcoh{m-1}{n+m}-m(m+k)\cdot\Pcoh{k+m-1}{n+m}\right)\,.
	\end{align*}
\end{cor}
\begin{rem}
	There are two reasons why the power sum basis is better suited for our computations than the elementary symmetric polynomial basis. First, power sums of Chern roots behave more simple with respect to the short exact sequences, see Proposition \ref{pro:q1}. Moreover, the power sum polynomials behave nicely under the Adams operations, see Proposition \ref{pro:AdamsKirwan}.
\end{rem}

To make the proof more readable we omit superscript in the notation of power sum elements $\Pcoh{}{}$. We use the notation $[l]$ for the set $\{1,\dots,l\}$. We write $\lambda_A,\, k$ for a sequence $\lambda_A$ with one added element $k$, i.e.
$$\Pcoh{{\lambda_A},\, k}{}:=\Pcoh{{\lambda_A}}{}\cdot\Pcoh{k}{} \,.$$
We need the following lemma. 
\begin{lemma} \label{lem:qm}
	Let $\lambda=(\lambda_1,\dots,\lambda_l)$ be a sequence of non-negative integers and $C\subseteq [l]$ a subset. For an arbitrary positive number $m$ we have
	$$\sum_{A\subseteq C}\left(m^{|A|}\cdot l(A)\right)=\sum_{A\subseteq C}\left(m^{|A|+1} \cdot l(C\setminus A)\right)=l(C)\cdot m(1+m)^{|C|-1}\,.$$
\end{lemma}
\begin{proof}
	we have
	\begin{multline*}
		\sum_{A\subseteq C}m^{|A|}\cdot l(A)=
		\sum_{A\subseteq C} \Big(m^{|A|}\cdot\sum_{i\in A} \lambda_i\Big)
		=\sum_{i\in C} \Big(\lambda_i\cdot\sum_{i\in A\subseteq C} m^{|A|}\Big)=\\
		=\Big(\sum_{i\in C} \lambda_i\Big)\cdot\Big(\sum_{A'\subseteq C\setminus \{*\}} m^{|A'|+1}\Big) 
		= l(C)\cdot m\cdot (1+m)^{|C|-1}\,.
	\end{multline*}
	On the other hand
	\begin{multline*}
		\sum_{A\subseteq C}m^{|A|+1}\cdot l(C\setminus A)=
		\sum_{A\subseteq C} \Big(m^{|A|}\cdot\sum_{i\notin A} \lambda_i\Big)
		=\sum_{i\in C} \Big(\lambda_i\cdot\sum_{i\notin A\subseteq C} m^{|A|+1}\Big)=\\
		=\Big(\sum_{i\in C} \lambda_i\Big)\cdot\Big(\sum_{A'\subseteq C\setminus \{*\}} m^{|A'|+1}\Big) 
		= l(C)\cdot m\cdot (1+m)^{|C|-1}\,. \qedhere
	\end{multline*}
\end{proof}
\begin{proof}[Proof of Theorem \ref{tw:qm}]
	We proceed by induction on $m$. For $m=1$ theorem simplifies to Proposition \ref{tw:q1}. Suppose that the theorem holds for $m$. We will prove that it holds also form $m+1$. \\
	We want to use Theorem \ref{tw:Evain}. First, let us consider the summand $\rho\circ \qq_{m}(\Pcoh{\lambda}{})$. The inductive assumption implies that it is equal to
	\begin{align} \label{eq:qk1}
		(-1)^{m+1}\cdot\sum_{A\subseteq[l]}(-1)^{|A|}m^{|A|}(l(A)+m)\cdot\rho\big(\Pcoh{{\lambda_A,\, l(A)+m-1}}{}\big)\,.
	\end{align}
	We use Proposition \ref{pro:rho} to compute summands $\rho(\Pcoh{{\lambda_A},\, l(A)+m-1}{})$. It yields a sum indexed by subsets of the set
	$$([l]\setminus A)\cup \{\infty\}\,,$$
	where $\{\infty\}$ corresponds to factor $\Pcoh{l(A)+m-1}{}$.
	Every such subset $\tilde{B}$ is uniquely determined by a subset $B\subseteq[l]$ such that $A\cap B=\varnothing$ and information whether the additional point $\{\infty\}$ belongs to $\tilde{B}$.
	Therefore, the sum computing $\rho(\Pcoh{{\lambda_A},l(A)+m-1}{})$ may be split into two sums indexed by  the set
	$$\{B\subseteq [l]|\, A\cap B=\varnothing\}\,. $$
	We substitute Proposition \ref{pro:rho} to equation \eqref{eq:qk1} and perform the mentioned splitting. The sum corresponding to $\infty \notin \tilde{B}$ is of the form
	\begin{align} \label{eq:qk2}
		(-1)^{m+1}\cdot\sum_{\substack{A,B\subseteq [l],\\ A\cap B=\varnothing}}
		(-1)^{|A|+|B|+1}m^{|A|}(l(A)+m)(l(B)+2)\cdot
		\Pcoh{{\lambda_{A\cup B},\, l(B)+2,\, l(A)+m}}{}\,.
	\end{align}
	The second sum, corresponding to $\infty \in \tilde{B}$ is equal to
	\begin{align}\label{eq:qk3}
		(-1)^{m+1}\cdot\sum_{\substack{A,B\subseteq [l] \\ A\cap B=\varnothing}}
		(-1)^{|A|+|B|+2}m^{|A|}(l(A)+m)(l(A)+l(B)+m+1)\cdot
		\Pcoh{{\lambda_{A\cup B},\, l(A)+l(B)+m}}{}\,.
 \end{align}
	We have $$\rho\circ \qq_{m}(\Pcoh{\lambda}{})=\eqref{eq:qk2}+\eqref{eq:qk3}\,.$$
	We apply analogous procedure to the expression $\qq_m\circ \rho(\Pcoh{\lambda}{})$. By Proposition \ref{pro:rho} it is equal to
	$$\sum_{B\subseteq [l]}(-1)^{|B|+1}(l(B)+2)\cdot
	\qq_m\big(\Pcoh{{\lambda_B,\, l(B)+1}}{}\big)\,.$$
	We use inductive assumption and split the obtained sum into two parts. The one corresponding to subsets not containing the additional point is equal to
	\begin{align} \label{eq:qk4}
		(-1)^{m+1}\sum_{\substack{A,B\subseteq [l],\\ A\cap B=\varnothing}}
		(-1)^{|A|+|B|+1}m^{|A|}(l(A)+m)(l(B)+2)\cdot
		\Pcoh{{\lambda_{A\cup B},\, l(B)+2,\, l(A)+m}}{}\,.
	\end{align}
	The other one is equal to
	\begin{align} \label{eq:qk5}
		(-1)^{m+1}\cdot\sum_{\substack{A,B\subseteq [l], \\ A\cap B=\varnothing}}
		(-1)^{|A|+|B|+2}m^{|A|+1}(l(B)+2)(l(A)+l(B)+m+1)\cdot
		\Pcoh{{\lambda_{A\cup B},\, l(A)+l(B)+m}}{}\,.
	\end{align}
	We have
	$$\rho\circ \qq_{m}(\Pcoh{\lambda}{})=\eqref{eq:qk4}+\eqref{eq:qk5}\,.$$
	The summands $\eqref{eq:qk4}$ and $\eqref{eq:qk2}$ are identical. By Theorem \ref{tw:Evain} we need to compute 
	\begin{align} \label{eq:qk6}
		m\cdot\qq_{m+1}(\Pcoh{\lambda}{})=\rho\circ \qq_{m}(\Pcoh{\lambda}{})- \rho\circ \qq_{m}(\Pcoh{\lambda}{})=\eqref{eq:qk3}-\eqref{eq:qk5}\,.
	\end{align}
	Grouping terms with the same $A\cup B$ we obtain that it is equal to
	$$
	(-1)^{m+1}\sum_{C\subseteq [l]}\Big(
	(-1)^{|C|}(l(C)+m+1)
	\Pcoh{{\lambda_{C},\, l(C)+m}}{}
	\sum_{A\subseteq C}
	\big(m^{|A|}(l(A)+m)-m^{|A|+1}(l(B)+2)\big)\Big)
	$$
	where $B=C\setminus A$. Lemma \ref{lem:qm} implies that
	$$\sum_{\substack{A\subseteq C,\\ B=C\setminus A}}
	\left(m^{|A|}l(A)-m^{|A|+1}l(B)\right)=0\,.$$
	On the other hand
	$$\sum_{A\subseteq C}
	\left(m^{|A|+1}-2m^{|A|+1}\right)=
	(-m)\cdot\sum_{A\subseteq C}m^{|A|}=(-m)(1+m)^{|C|}\,.$$
	The theorem follows from formula \eqref{eq:qk6}.
\end{proof}

\appendix
\section{Combinatorics of rational functions}

Let $\lambda$ be a partition of $n$ and let $(k,l)$ be its corner.  We introduced the notation (cf. Section \ref{s:pushK})
\[r_{\lambda, (k,l)}:= \frac{\eu(T_\lambda\H_{n+1})}{\eu(T_{\lambda,(k,l)}\H_{n,n+1})} \in \ZZ(q,t)\,. \]
Both Euler classes in the expression above are easy to compute - they are products of the form $\prod_{(i,j)} (1-q^{-i} t^{-j})$, where $(i,j)$ are tangent weights. Recall from Section \ref{prel:hilb} that most of the tangent weights at the fixed point $\lambda$ are identical in $\Hilb{n+1}$ and in $\Hilb{n, n+1}$, so factors corresponding to identical weights cancel out. What is left in the numerator are the tangent weights at $\Hilb{n+1}$ which are not tangent weights in the nested Hilbert scheme - each of them corresponds to a box below or to the left of the corner $(k,l)$. For boxes below the corner, the numerator weights are of the form $(-a, b+1)$ and the denominator weights of the form $(-a,b)$, while for the boxes to the left the numerator weights are $(a+1,-b)$ and the denominator weights are $(a,-b)$.
Hence $r_{\lambda, (k,l)}$ decomposes into the following product:

\[ r_{\lambda, (k,l)}= \prod_{d=0}^{l-1} \frac{1- q^{a_d}t^{-b_d-1} }{1-q^{a_d}t^{-b_d}}  \prod_{s=0}^{k-1}  \frac{1- q^{-a_s-1}t^{b_s} }{1-q^{-a_s}t^{b_s}} = \prod_{d=0}^{l-1} \frac{1}{t} \cdot \frac{q^{a_d}- t^{b_d+1}}{q^{a_d}- t^{b_d}} \prod_{s=0}^{k-1} \frac{1}{q} \cdot  \frac{q^{a_s+1}-t^{b_s}}{q^{a_s}-t^{b_s}}\,. \]
where $a_d=a_{\lambda, (k,d)}$, $b_d=b_{\lambda, (k,d)}$, $a_s=a_{\lambda, (s,l)}$, $b_s=b_{\lambda, (s,l)}$. Factoring out the $\frac{1}{q}, \frac{1}{t}$ terms one gets
\[ r_{\lambda, (k,l)}= \frac{1}{q^k t^l}  \prod_{d=0}^{l-1} \frac{q^{a_d}-t^{b_d+1}}{q^{a_d}-t^{b_d}} \prod_{s=0}^{k-1} \frac{q^{a_s+1}-t^{b_s}}{q^{a_s}-t^{b_s}}\,.\]
\begin{rem}
	In the case $k=0$ or $l=0$ we use convention $\prod_{s=0}^{-1}(\dots)=1\,.$
\end{rem}
Let us introduce two functions, corresponding to the two types of products which appear in the expression for $r_{\lambda, (k,l)}$.

\begin{adf}
	For a pair of non-negative integers $(a,b)$, we define rational functions
	\begin{align*}
		W_{a,b}(q,t)&:=  \frac{q^{a+1}-t^b}{q^a-t^b}\,,&
		U_{a,b}(q,t)&:=  \frac{q^{a}-t^{b+1}}{q^a-t^b}\,.&
	\end{align*}
Given a nonempty partition $\lambda$ and a box $(k,l)$ in the Young diagram of $\lambda$, we let
 \begin{align*}
		 	W_{\lambda,(k,l)}(q,t):= & W_{a_{\lambda,(k,l)},b_{\lambda,(k,l)}}(q,t)\,,&
		 	U_{\lambda,(k,l)}(q,t):= & U_{a_{\lambda,(k,l)},b_{\lambda,(k,l)}}(q,t)\,.&
		\end{align*}

\end{adf}
Let us note some basic properties of these functions, all of which are checked by an easy computation.
\begin{pro} \label{pro:lim}
	\begin{enumerate}
		\item We have
		$$W_{a,b}(q,t) =U_{b,a}(t,q)\,. $$
		\item For $a \neq 0$, we have
		\begin{align*}
			\lim_{q \to 0} W_{a,b}(q,t) &=1\,,&
			\lim_{q \to 0} U_{a,b}(q,t) &=t\,,&
			\lim_{q \to \infty} U_{a,b}(q,t) &=1\,.
		\end{align*}
		\item We have
		\begin{align*}
			W_{a,b}(q,t)&=qW_{a,b}(q^{-1},t^{-1})\,,&
			U_{a,b}(q,t)&=tU_{a,b}(q^{-1},t^{-1})\,.
		\end{align*}
	\end{enumerate}
\end{pro}

\begin{adf} \label{df:R}
	Let $\lambda$ be a nonempty partition and let $(k,l)$ be a corner in the Young diagram of $\lambda$. Let
	 $$R_{\lambda,(k,l)}(q,t):=
	  \prod_{d=0}^{l-1} U_{\lambda,(k,d)}(q,t)
	 \cdot
	 \prod_{s=0}^{k-1} W_{\lambda,(s,l)}(q,t) \,.$$
\end{adf}
\begin{pro} \label{pro:1/x}
	Let $\lambda$ be a nonempty partition and $(k,l)$ be its corner.
	We have
	$$R_{\lambda,(k,l)}(q,t)=q^k t^l \cdot R_{\lambda,(k,l)}(q^{-1},t^{-1}) \,.$$
\end{pro}
\begin{proof}
	It follows directly from Proposition \ref{pro:lim} (3).
\end{proof}

Note that the function $R_{\lambda,(k,l)}$ is the same as $\tilde{r}_{\lambda,(k,l)}$ from Section \ref{s:pushK}, which in turn is equal to $r_{\lambda,(k,l)}$ rescaled by the factor $q^kt^l$. We have
$$
	\tilde{r}_{\lambda,(k,l)}=R_{\lambda,(k,l)}(q,t) \,, \qquad r_{\lambda,(k,l)} = \frac{1}{q^k t^l}\cdot R_{\lambda,(k,l)}(q,t)
$$

The remaining part of the appendix is devoted to proving some technical results about the function $R_{\lambda,(k,l)}$, which are used throughout the paper.

\begin{cor}\label{cor:Apinverse}
	Consider the situation as in Proposition \ref{pro:1/x}. Let $\tau:\KTh_\TT(pt)\to\KTh_\TT(pt)$ be a map given by $\tau(t)=t^{-1}$ and $\tau(q)=q^{-1}$. We have
$$	\tau(\tilde{r}_{\lambda,(k,l)}\big)=q^{-k}t^{-l}\cdot\tilde{r}_{\lambda,(k,l)}$$
\end{cor}

\begin{pro} \label{pro:Ap1}
	Let $\lambda$ be a nonempty partition and let $(k_1,l_1),\dots,(k_N,l_N)$ be its corners sorted from the uppermost one to the lowest one. Set $l_{N+1}=-1$. Then
	$$\lim_{q\to 0}R_{\lambda,(k_i,l_i)}(q,t)=\sum_{i=l_{i-1}+1}^{l_{i}} t^i \,. $$
\end{pro} 
\begin{proof}
	By definition
	\begin{align*}
		\lim_{q\to 0} R_{\lambda,(k_i,l_i)}(q,t) 
		&=\lim_{q\to 0}\prod_{s=0}^{k_i-1} W_{\lambda,(s,l_i)}(q,t)
		\cdot
		\lim_{q\to 0}\prod_{s=0}^{l_{i-1}} U_{\lambda,(k_i,s)}(q,t)
		\cdot
		\lim_{q\to 0}\prod_{s=l_{i-1}+1}^{l_i-1} U_{\lambda,(k_i,s)}(q,t)
\end{align*}
	Proposition \ref{pro:lim} (2) implies that the first limit is equal to $1$ and the second to $t^{l_{i-1}+1}$. The third rational function does not depend on the variable $q$. Thus
	\[
		\lim_{q\to 0} R_{\lambda,(k_i,l_i)}(q,t) =  t^{l_{i-1}+1}\cdot \prod_{s=1}^{l_{i}-l_{i-1}-1} U_{1,s}(q,t) 
		=t^{l_{i-1}+1}\cdot\frac{1-t^{l_{i}-l_{i-1}}}{1-t}=\sum_{i=l_{i-1}+1}^{l_{i}} t^i\,. \qedhere
	\]
\end{proof}	
\begin{cor} \label{cor:Aplim1}
	Consider the situation as in Proposition \ref{pro:Ap1}. Then
	$$\lim_{q\to 0}\tilde{r}_{\lambda,(k_i,l_i)}=\sum_{i=l_{i-1}+1}^{l_{i}} t^i \,. $$
\end{cor}
	\begin{pro} \label{pro:Ap2}
		Let $\lambda$ be a nonempty partition and $\tilde{\lambda}$ the partition $\lambda$ without the first column. Suppose that $(k,l)$ is a corner of $\lambda$ such that $k\neq 0$. Then the limit
		$$\lim_{q \to \infty} \big(R_{\lambda,(k,l)}(q,t)-qR_{\tilde{\lambda},(k-1,l)}(q,t)\big) $$
		exists.
	\end{pro}
	\begin{proof}
		By definition
		$$R_{\lambda,(k,l)}(q,t)=R_{\tilde{\lambda},(k-1,l)}(q,t)\cdot W_{k,b_{\lambda,(0,l)}}(q,t)\,. $$
		Let $b:=b_{\lambda,(0,l)}$. It follows that 
		\begin{align*}
			R_{\lambda,(k,l)}(q,t)-qR_{\tilde{\lambda},(k-1,l)}(q,t)&=
			R_{\tilde{\lambda},(k-1,l)}(q,t)\cdot (W_{k,b}-q)\\
			&=
			R_{\tilde{\lambda},(k-1,l)}(q,t)\cdot \frac{t^{b}\cdot(q-1)}{q^k-t^b}\\
			&= R_{\tilde{\lambda},(k-1,l)}(q^{-1},t^{-1}) \cdot \frac{q^{k-1}t^{b+l}\cdot(q-1)}{q^k-t^b}\,,
		\end{align*}
		where the last equality follows from Proposition \ref{pro:1/x}. Proposition \ref{pro:Ap1} implies that the limit
		$$\lim_{q\to \infty} R_{\tilde{\lambda},(k-1,l)}(q^{-1},t^{-1}) $$
		exists. The limit of the second factor also exists (it is equal to $t^{b+l}$).
	\end{proof}
\begin{cor} \label{cor:Aplim2}
	Consider the situation as in Proposition \ref{pro:Ap2}. Then the limit
	$$\lim_{q \to \infty} \big(\tilde{r}_{\lambda,(k,l)}-q\tilde{r}_{\tilde{\lambda},(k-1,l)}\big) $$
	exists.
\end{cor}

\begin{pro} \label{pro:Ap3}
	Let $\lambda$ be a nonempty partition and $(0,l)$ be its corner. Then the limit
	$$\lim_{q \to \infty}R_{\lambda,(0,l)}(q,t)$$
	exists.
\end{pro}
\begin{proof}
	By Proposition \ref{pro:1/x} we have
	$$R_{\lambda,(0,l)}(q,t)=t^lR_{\lambda,(0,l)}(q^{-1},t^{-1})\,.$$
	The limit of the right hand side exists thanks to Proposition \ref{pro:Ap1}.
\end{proof}
\begin{cor} \label{cor:Aplim3}
	Consider the situation as in Proposition \ref{pro:Ap3}. Then the limit
	$$\lim_{q \to \infty} \tilde{r}_{\lambda,(0,l)}(q,t) $$
	exists.
\end{cor}

\section{Declarations} The authors have no conflict of
interest to declare that are relevant to this article.

\bibliographystyle{alpha}
\bibliography{Hilb}

\begin{thebibliography}{FRW21}

\bibitem[AB84]{AB}
M.~F. Atiyah and R.~Bott.
\newblock The moment map and equivariant cohomology.
\newblock {\em Topology}, 23(1):1--28, 1984.

\bibitem[Ada62]{Adams}
J.~F. Adams.
\newblock Vector fields on spheres.
\newblock {\em Ann. of Math. (2)}, 75:603--632, 1962.

\bibitem[BB73]{B-B1}
A.~Bia{\l}ynicki-Birula.
\newblock Some theorems on actions of algebraic groups.
\newblock {\em Ann. of Math. (2)}, 98:480--497, 1973.

\bibitem[BG17]{BG}
Gwyn Bellamy and Victor Ginzburg.
\newblock {$SL_2$}-action on {H}ilbert schemes and {C}alogero-{M}oser spaces.
\newblock {\em Michigan Math. J.}, 66(3):519--532, 2017.

\bibitem[BKR01]{BKR}
T.~Bridgeland, A.~King, and M.~Reid.
\newblock The {M}c{K}ay correspondence as an equivalence of derived categories.
\newblock {\em J. Amer. Math. Soc.}, 14(3):535--554, 2001.

\bibitem[Boi05]{Boissiere}
Samuel Boissi\`ere.
\newblock Chern classes of the tangent bundle on the {H}ilbert scheme of points
  on the affine plane.
\newblock {\em J. Algebraic Geom.}, 14(4):761--787, 2005.

\bibitem[Boi06a]{Boissiere3}
S.~Boissi\`ere.
\newblock Towards the multiplicative behavior of the {$K$}-theoretical
  {M}c{K}ay correspondence.
\newblock {\em Math. Z.}, 252(3):533--555, 2006.

\bibitem[Boi06b]{Boissiere2}
Samuel Boissi\`ere.
\newblock On the {M}c{K}ay correspondences for the {H}ilbert scheme of points
  on the affine plane.
\newblock {\em Math. Ann.}, 334(2):419--438, 2006.

\bibitem[BV82]{BV}
Nicole Berline and Mich\`ele Vergne.
\newblock Classes caract\'{e}ristiques \'{e}quivariantes. {F}ormule de
  localisation en cohomologie \'{e}quivariante.
\newblock {\em C. R. Acad. Sci. Paris S\'{e}r. I Math.}, 295(9):539--541, 1982.

\bibitem[Car10]{Carl}
Erik Carlsson.
\newblock Vertex operators, {G}rassmannians, and {H}ilbert schemes.
\newblock {\em Comm. Math. Phys.}, 300(3):599--613, 2010.

\bibitem[CE15]{Evain}
Pierre-Emmanuel Chaput and Laurent Evain.
\newblock On the equivariant cohomology of {H}ilbert schemes of points in the
  plane.
\newblock {\em Ann. Inst. Fourier (Grenoble)}, 65(3):1201--1250, 2015.

\bibitem[CG10]{CG}
Neil Chriss and Victor Ginzburg.
\newblock {\em Representation theory and complex geometry}.
\newblock Modern Birkh\"{a}user Classics. Birkh\"{a}user Boston, Ltd., Boston,
  MA, 2010.
\newblock Reprint of the 1997 edition.

\bibitem[Che96]{Cheah2}
Jan Cheah.
\newblock On the cohomology of {H}ilbert schemes of points.
\newblock {\em J. Algebraic Geom.}, 5(3):479--511, 1996.

\bibitem[Che98]{Cheah}
Jan Cheah.
\newblock Cellular decompositions for nested {H}ilbert schemes of points.
\newblock {\em Pacific J. Math.}, 183(1):39--90, 1998.

\bibitem[EG00]{EGGRR}
Dan Edidin and William Graham.
\newblock Riemann-{R}och for equivariant {C}how groups.
\newblock {\em Duke Math. J.}, 102(3):567--594, 2000.

\bibitem[ES87]{ESBBdec}
Geir Ellingsrud and Stein~Arild Str{\o}mme.
\newblock On the homology of the {H}ilbert scheme of points in the plane.
\newblock {\em Invent. Math.}, 87(2):343--352, 1987.

\bibitem[ES93]{ES}
Geir Ellingsrud and Stein~Arild Str{\o}mme.
\newblock Towards the {C}how ring of the {H}ilbert scheme of {${\bf P}^2$}.
\newblock {\em J. Reine Angew. Math.}, 441:33--44, 1993.

\bibitem[Eva04]{Evain2}
Laurent Evain.
\newblock Irreducible components of the equivariant punctual {H}ilbert schemes.
\newblock {\em Adv. Math.}, 185(2):328--346, 2004.

\bibitem[FL85]{RR}
W.~Fulton and S.~Lang.
\newblock {\em Riemann-{R}och algebra}, volume 277 of {\em Grundlehren der
  mathematischen Wissenschaften [Fundamental Principles of Mathematical
  Sciences]}.
\newblock Springer-Verlag, New York, 1985.

\bibitem[Fog68]{Fogarty}
John Fogarty.
\newblock Algebraic families on an algebraic surface.
\newblock {\em American Journal of Mathematics}, 90(2):511--521, 1968.

\bibitem[FRW21]{FRW}
L.~Feh\'{e}r, R.~Rim\'{a}nyi, and A.~Weber.
\newblock Motivic {C}hern classes and {K}-theoretic stable envelopes.
\newblock {\em Proc. Lond. Math. Soc. (3)}, 122(1):153--189, 2021.

\bibitem[FT11]{FT}
Boris Feigin and Alexander Tsymbaliuk.
\newblock Equivariant {$K$}-theory of {H}ilbert schemes via shuffle algebra.
\newblock {\em Kyoto J. Math.}, 51(4):831--854, 2011.

\bibitem[GN15]{Gorsky}
Eugene Gorsky and Andrei Negu\c{t}.
\newblock Refined knot invariants and {H}ilbert schemes.
\newblock {\em J. Math. Pures Appl. (9)}, 104(3):403--435, 2015.

\bibitem[Gro96]{Gro}
Ian Grojnowski.
\newblock Instantons an affine algebras. {I}. {T}he {H}ilbert scheme and vertex
  operators.
\newblock {\em Math. Res. Lett.}, 3(2):275--291, 1996.

\bibitem[GS93]{GS}
Lothar G\"{o}ttsche and Wolfgang Soergel.
\newblock Perverse sheaves and the cohomology of {H}ilbert schemes of smooth
  algebraic surfaces.
\newblock {\em Math. Ann.}, 296(2):235--245, 1993.

\bibitem[GS05]{GorS1}
Iain Gordon and J.~Toby Stafford.
\newblock Rational {C}herednik algebras and {H}ilbert schemes.
\newblock {\em Adv. Math.}, 198(1):222--274, 2005.

\bibitem[Hai98]{Haiman2}
Mark Haiman.
\newblock {$t,q$}-{C}atalan numbers and the {H}ilbert scheme.
\newblock volume 193, pages 201--224. 1998.
\newblock Selected papers in honor of Adriano Garsia (Taormina, 1994).

\bibitem[Hai03]{Haiman}
Mark Haiman.
\newblock Combinatorics, symmetric functions, and {H}ilbert schemes.
\newblock In {\em Current developments in mathematics, 2002}, pages 39--111.
  Int. Press, Somerville, MA, 2003.

\bibitem[Hsi18]{YiNing}
Yi-Ning Hsiao.
\newblock Geometry and topology of refined structures on the hilbert scheme of
  points on the plane, 2018.
\newblock PhD thesis.

\bibitem[K{\"o}c98]{Koeck}
B.~K{\"o}ck.
\newblock The {Grothendieck-Riemann-Roch} theorem for group scheme actions.
\newblock {\em Annales scientifiques de l'\'Ecole Normale Sup\'erieure}, Ser.
  4, 31(3):415--458, 1998.

\bibitem[Leh99]{Lehn}
Manfred Lehn.
\newblock Chern classes of tautological sheaves on {H}ilbert schemes of points
  on surfaces.
\newblock {\em Invent. Math.}, 136(1):157--207, 1999.

\bibitem[LQ10]{Wei}
Wei-Ping Li and Zhenbo Qin.
\newblock Equivariant cohomology of incidence {H}ilbert schemes and infinite
  dimensional {L}ie algebras.
\newblock {\em Manuscripta Math.}, 133(3-4):519--544, 2010.

\bibitem[LS01]{Lehn2}
M.~Lehn and Ch. Sorger.
\newblock Symmetric groups and the cup product on the cohomology of {H}ilbert
  schemes.
\newblock {\em Duke Math. J.}, 110(2):345--357, 2001.

\bibitem[MN18]{QuiverKirwan}
Kevin McGerty and Thomas Nevins.
\newblock Kirwan surjectivity for quiver varieties.
\newblock {\em Invent. Math.}, 212(1):161--187, 2018.

\bibitem[MS05]{MS}
Ezra Miller and Bernd Sturmfels.
\newblock {\em Combinatorial commutative algebra}, volume 227 of {\em Graduate
  Texts in Mathematics}.
\newblock Springer-Verlag, New York, 2005.

\bibitem[Nak96]{Nak2}
Hiraku Nakajima.
\newblock Jack polynomials and hilbert schemes of points on surfaces, 1996.
\newblock arXiv:alg-geom/9610021.

\bibitem[Nak99]{Nak1}
Hiraku Nakajima.
\newblock {\em Lectures on {H}ilbert schemes of points on surfaces}, volume~18
  of {\em University Lecture Series}.
\newblock American Mathematical Society, Providence, RI, 1999.

\bibitem[NY11]{NY}
Hiraku Nakajima and Kota Yoshioka.
\newblock Perverse coherent sheaves on blow-up. {II}. {W}all-crossing and
  {B}etti numbers formula.
\newblock {\em J. Algebraic Geom.}, 20(1):47--100, 2011.

\bibitem[OR18]{OR}
Alexei Oblomkov and Lev Rozansky.
\newblock Knot homology and sheaves on the {H}ilbert scheme of points on the
  plane.
\newblock {\em Selecta Math. (N.S.)}, 24(3):2351--2454, 2018.

\bibitem[Prz20]{Przez}
Tomasz Prze\'{z}dziecki.
\newblock The combinatorics of {$\Bbb C^\ast$}-fixed points in generalized
  {C}alogero-{M}oser spaces and {H}ilbert schemes.
\newblock {\em J. Algebra}, 556:936--992, 2020.

\bibitem[PT19]{Pandha}
Rahul Pandharipande and Hsian-Hua Tseng.
\newblock The {H}ilb/{S}ym correspondence for {$\Bbb C^2$}: descendents and
  {F}ourier-{M}ukai.
\newblock {\em Math. Ann.}, 375(1-2):509--540, 2019.

\bibitem[Smi20]{Smirnov}
Andrey Smirnov.
\newblock Elliptic stable envelope for {H}ilbert scheme of points in the plane.
\newblock {\em Selecta Math. (N.S.)}, 26(1):Paper No. 3, 57, 2020.

\bibitem[SV13]{SV}
Olivier Schiffmann and Eric Vasserot.
\newblock The elliptic {H}all algebra and the {$K$}-theory of the {H}ilbert
  scheme of {$\Bbb A^2$}.
\newblock {\em Duke Math. J.}, 162(2):279--366, 2013.

\bibitem[Tho92]{Tho}
R.~W. Thomason.
\newblock Une formule de lefschetz en {K}-th{\'e}orie {\'e}quivariante
  alg{\'e}brique.
\newblock {\em Duke Math. Journal}, 68:447--462, 1992.

\end{thebibliography}

\end{document}